\documentclass[smallextended,numbook,runningheads]{svjour3}

\usepackage{mathptmx}      

\smartqed
\usepackage[pdfborder={0 0 0}]{hyperref}
\usepackage{amsmath}
\usepackage{amssymb}
\usepackage{graphicx}
\usepackage[dvipsnames]{xcolor}
\usepackage[english]{babel}
\usepackage{url}
\usepackage{nicefrac}
\usepackage{upquote}
\DeclareMathAlphabet{\mathbf}{OT1}{cmr}{bx}{it}

\def\be#1\ee{\begin{equation}#1\end{equation}}
\newcommand{\bea}{\begin{eqnarray}}
\newcommand{\eea}{\end{eqnarray}}
\newcommand{\beas}{\begin{eqnarray*}}
\newcommand{\eeas}{\end{eqnarray*}}

\newcommand{\vnull}{\boldsymbol{0}}

\usepackage{calc}
\usepackage{algorithm}
\usepackage{algorithmic}
\makeatletter
\@addtoreset{algorithm}{section}
\makeatother
\makeatletter
\def\@hspace#1{\begingroup\setlength\dimen@{#1}\hskip\dimen@\endgroup}
\makeatother

\algsetup{linenodelimiter=.}

\renewcommand{\d}{\,\mathrm{d}}

\newcommand{\rev}[1]{{\color{black}{#1}}}

\newcommand{\diag}{\mathrm{diag}}

\title{Model order reduction of layered waveguides \\  
via rational Krylov fitting}

\author{Vladimir Druskin 
  \and Stefan G\"uttel
  \and Leonid~Knizhnerman}

\date{Received: date / Accepted: date}

\institute{Vladimir Druskin \at
              Department of Mathematical Sciences \\
              Worcester Polytechnic Institute\\
              100 Institute Rd\\
              Worcester, MA 01609, USA\\
              \email{vdruskin1@gmail.com}           
           \and
           Stefan G\"uttel \at
              Department of Mathematics\\
              The University of Manchester\\
              Alan Turing Building\\
              Manchester, M13\,9PL, UK\\
              \email{stefan.guettel@manchester.ac.uk}
           \and
           \rev{Leonid Knizhnerman  \at
           Marchuk Institute of Numerical Mathematics\\ 
           Russian Academy of Sciences\\
           Gubkin St.~8\\
           Moscow 119333, Russia\\
           \email{lknizhnerman@gmail.com}}
}

\journalname{}

\begin{document}

\setcounter{page}{1} 

        \bibliographystyle{plain}
        \maketitle

\begin{abstract}
Rational approximation  recently   emerged as an efficient numerical tool for    the solution of exterior wave propagation  problems.   Currently, this technique  is limited to wave media which are invariant along the main propagation direction.   We propose a new model order reduction-based  approach for compressing unbounded waveguides with layered inclusions.
It is based on the solution of a nonlinear rational least squares problem using the RKFIT method. We show that approximants can be converted into an accurate  finite difference representation within a rational Krylov framework. Numerical experiments indicate that RKFIT computes more accurate grids than previous analytic approaches and even works in the presence of pronounced scattering resonances. Spectral adaptation effects allow for finite difference grids with  dimensions near or even below the Nyquist limit.

\end{abstract}




\section{Introduction}




In this work we present a new  approach to the compression of Dirichlet-to-Neumann (DtN) maps of infinite waveguides with layered inclusions. This approach is inspired by rational approximation techniques from model order reduction (see, e.g., \cite{beattie2017model}), in this case  the RKFIT algorithm for nonlinear rational approximation~\cite{BG15}. 
As a prototypical problem we consider the infinite finite difference (FD) scheme 
\begin{subequations}\label{eq:MRelv}
\begin{eqnarray}
 2h^{-1}\left[ h^{-1}({\mathbf{u}_1 - \mathbf{u}_0}) + \mathbf b \right] &=&  (A+c_0 I) \mathbf{u}_0  \label{eq:MRel1v} \\
 h^{-1}\left[
 h^{-1}({\mathbf{u}_{j+1} - \mathbf{u}_j})  - h^{-1}({\mathbf{u}_{j} - \mathbf{u}_{j-1}})
   \right] &=&   (A+c_j I) \mathbf{u}_j ,
   \ \   j = 1,2,\ldots \label{eq:MRel2v}
\end{eqnarray}
\end{subequations}
where either $\mathbf{u}_0\in\mathbb{C}^N$ or $\mathbf{b}\in\mathbb{C}^N$ is given, $A\in\mathbb{C}^{N\times N}$ is  Hermitian, $c_j=0$ for all $j>L$, and the solution $\{\mathbf{u}_j\}_{j=0}^\infty\subset \mathbb{C}^N$ is assumed to be bounded.  
This problem arises  from the FD discretization of the three-dimensional (indefinite) Helmholtz equation 
$$\nabla^2 u + (k_\infty^2 - c(x)) u = 0$$ 
for $(x,y,z)\in [0,+\infty)\times [0,1]\times [0,1]$ with a compactly supported \emph{offset function} $c(x)$ for the wave number~$k_\infty$ and appropriate boundary conditions.  
Here, the matrix $A$ corresponds to the discretization of the transverse differential operator $-\partial_{yy}^2 - \partial_{zz}^2 - k_\infty^2$ at $x=0$ and is Hermitian indefinite. The variation of the wave number in the $x$-direction is  modelled by varying coefficients $c_j$, with the ``effective'' wave number  $\sqrt{k_\infty^2-c_j}$ at each grid point. 
The DtN operator $F$ for \eqref{eq:MRelv} is defined by the relationship $F \mathbf u_0=\mathbf b.$

Since \eqref{eq:MRelv} is a linear recurrence, $F=f_h(A)$ is a matrix function in~$A$. If $c_j\equiv 0$, the DtN function for  \eqref{eq:MRelv} at $x=0$ is $f_h(\lambda) = \sqrt{\lambda + (h\lambda/2)^2}$. As $h\to 0$ we obtain the DtN function $f(\lambda) = \sqrt{\lambda}$ for the continuous problem. In this case, a near-optimal rational approximant  to $f$ can be constructed analytically \cite{druskin1999gaussian,ingerman2000optimal,DGK16}.
More precisely, let the eigenvalues of  $A$ be contained in the union of two intervals $K = [a_1,b_1] \cup [a_2,b_2]$  with $a_1<b_1<0<a_2<b_2$. 
Then \cite{DGK16} gives an explicit construction of a compound Zolotarev rational function $r_n^{(Z)}$ of type $(n,n-1)$ such that 
\begin{equation}\label{eq:factR}
  \max_{\lambda \in K} | 1 -   r_n^{(Z)}(\lambda)/f(\lambda) | \asymp  \exp\left( - 2\pi^2n /{\log\left(256{a_1 b_2}/({a_2 b_1})\right)} \right) \  \text{as} \  n\to \infty
\end{equation}
for sufficiently large interval ratios $a_1/b_1$ and $b_2/a_2$. 
It is also shown in  \cite{DGK16} that the convergence factor in \eqref{eq:factR} is optimal.  Hence, the approximation error 
$\| f(A) - r_n^{(Z)}(A) \|_2 \leq C  \max_{\lambda \in K} | 1 -   r_n^{(Z)}
(\lambda)/f(\lambda) | $
 decays exponentially  at the same optimal rate. Interestingly, the continued fraction form of  $r_n^{(Z)}$ gives rise to a geometrically meaningful three-point FD scheme. {By ``geometrically meaningful'' we mean that the complex grid points align on a  curve in the complex plane which can be interpreted as a ``smooth'' deformation of the original $x$-coordinate axis.} 
 This is similar to the celebrated  \emph{perfectly matched layers} (PMLs)  which are  introduced via  complex coordinate stretching
 \cite{ChewWeedon,EM79,Ber94,Hagstrom}.

The analytic approach just outlined is  essentially limited to DtN functions such as $\sqrt{\lambda}$ and $\sqrt{\lambda + (h\lambda/2)^2}$. 
Here we aim to overcome this limitation by numerically computing a low-order rational approximant  $r_n(A)\approx f_h(A)$  and converting it into \rev{a} sparse representation in form of a three-point finite difference scheme.\footnote{Another recent  approach for compressing an NtD operator for the Helmholtz equation is based on randomized matrix probing   \cite{doi:10.1137/14095563X}. This approach  has the advantage of handling a rather wide class of multidimensional variable-coefficient problems at the expense of losing the sparse representation.} Our approach is applicable even in cases where the DtN map to be approximated is highly irregular due to the presence of scattering poles. 

 An illustrating example is given in Figure~\ref{fig:waveguide}, where the top panels show the amplitude/phase of the solution of a  waveguide problem on $[0,+\infty)\times [0,1]$, truncated and discretized by $300\times 150$ points. The step size is $h=1/150$ in both coordinate directions. For this problem we have chosen $k_\infty=14$ and $c_j = -9^2$ for the grid points $j=0,1,\ldots,L=150$. An absorbing boundary condition has been fitted to the right end of the domain to mimic the infinite extension $x\to\infty$. The modulus of the associated DtN function $f_h$ is shown in the bottom of Figure~\ref{fig:waveguide} (solid red curve). This function has several singularities between and close to the eigenvalues of the transverse FD matrix $A$ (the eigenvalue positions are indicated by the black dots). In particular, one eigenvalue $\lambda_j\approx 50.5$ is extremely close to a singularity of $f_h$, which can be associated with the near-resonance observed in the left portion of the waveguide. These singularities make it impossible to construct a uniform approximant $r_n\approx f_h$ over the negative and positive spectral subintervals of~$A$. 
Nevertheless, the RKFIT approximant $r_n$ of order $n=8$, also shown in the bottom of Figure~\ref{fig:waveguide} (dashed blue curve), has a relative accuracy $\| f_h(A) \mathbf{u}_0 - r_n(A)\mathbf{u}_0\|_2 /  \|f_h(A) \mathbf{u}_0\|_2 \approx 1.4\cdot 10^{-6}$ for the DtN map. We see that $r_n$ achieves this high accuracy by being close to $f_h$ in the vicinity of the eigenvalues of $A$, but not necessarily in between them. This remarkable \emph{spectral adaptation} is achieved without requiring a spectral decomposition of $A$ explicitly; RKFIT merely requires  matrix-vector products with the DtN map.

\begin{figure}[t]
\hspace*{3mm}\begin{minipage}{16cm}
\hspace*{-3mm}\includegraphics[width=6.1cm]{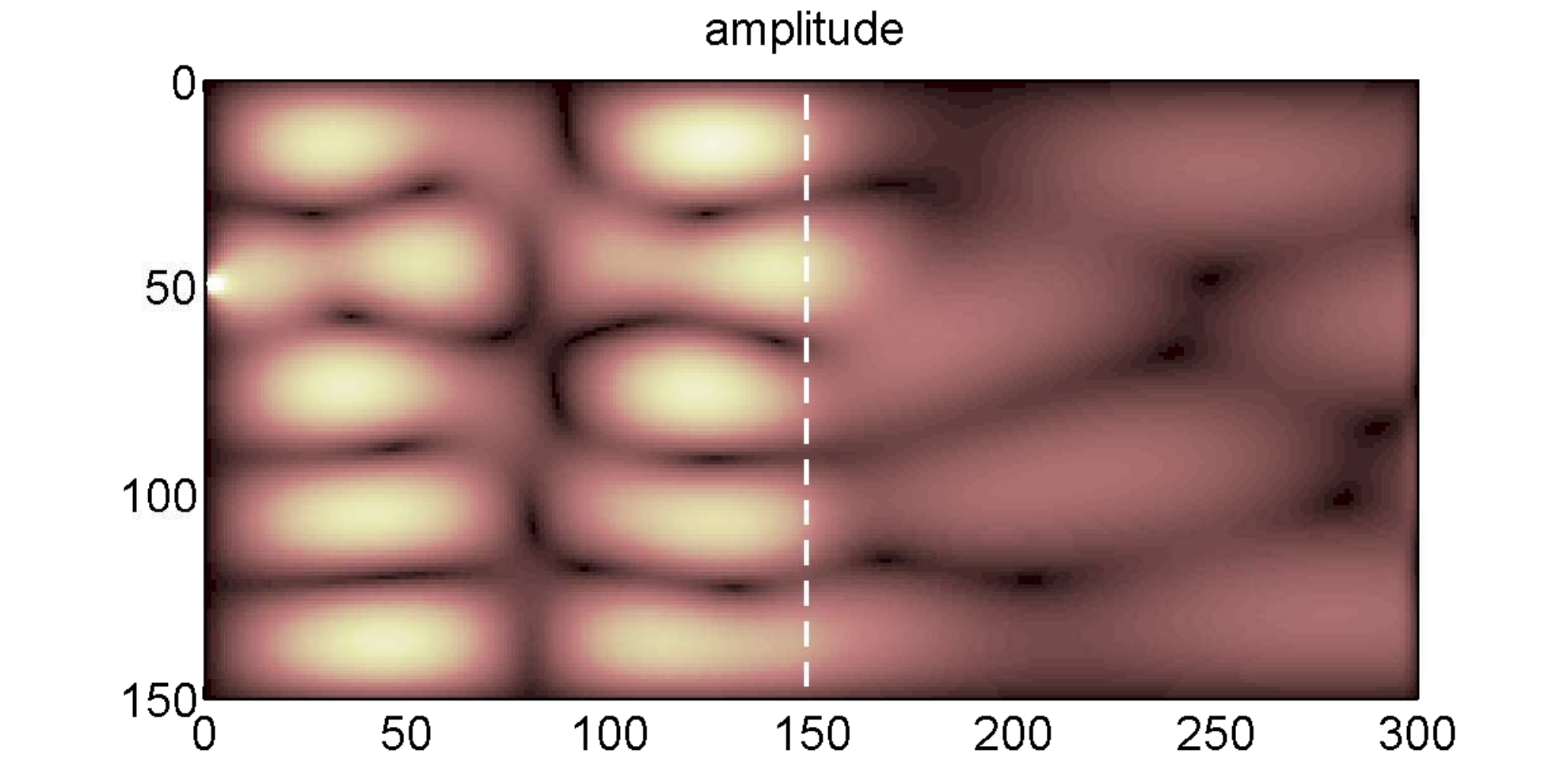}
\hspace*{-4mm}\includegraphics[width=6.1cm]{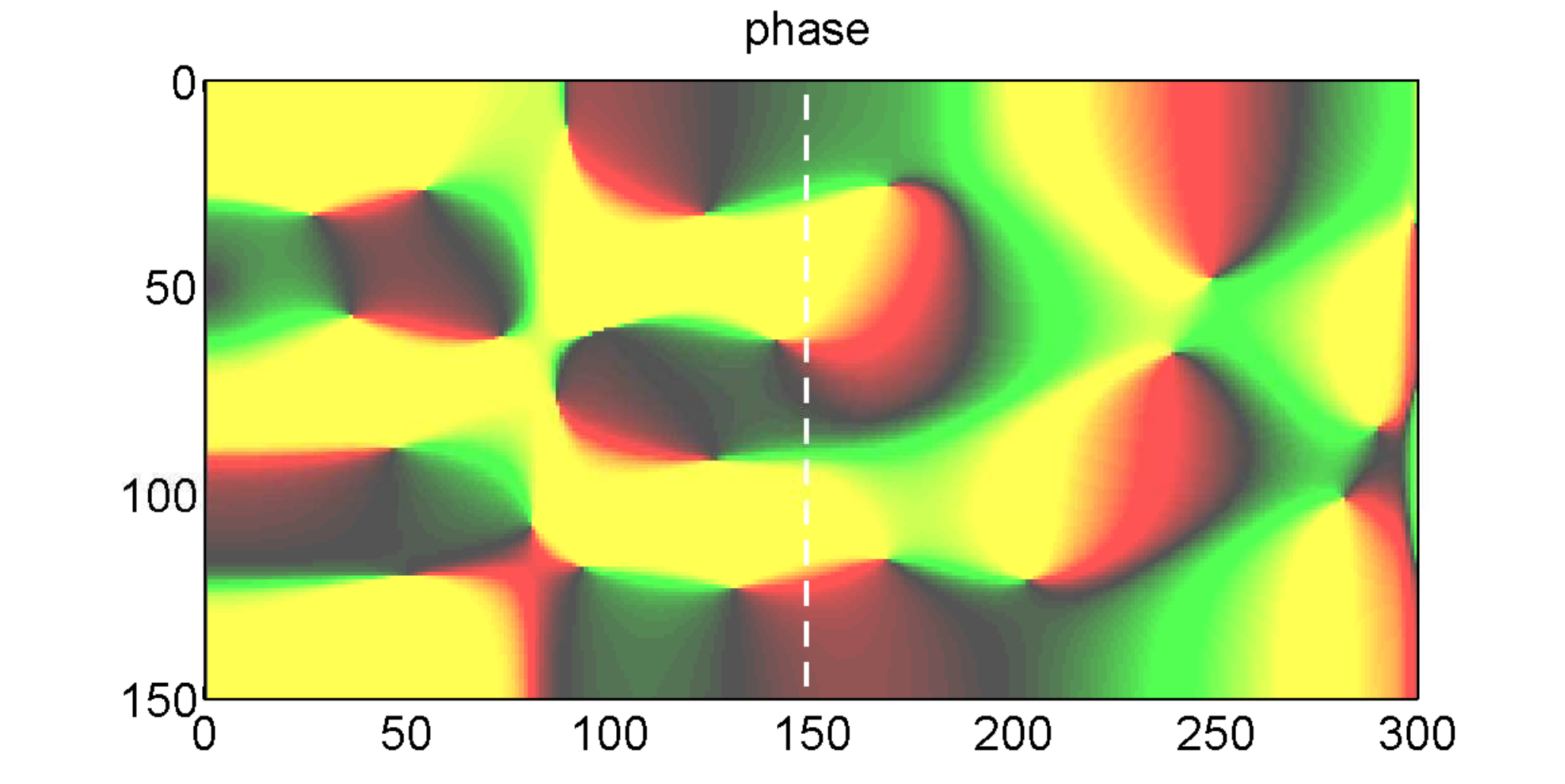}\\[2mm]
\hspace*{-13.5mm}\includegraphics[width=13.6cm]{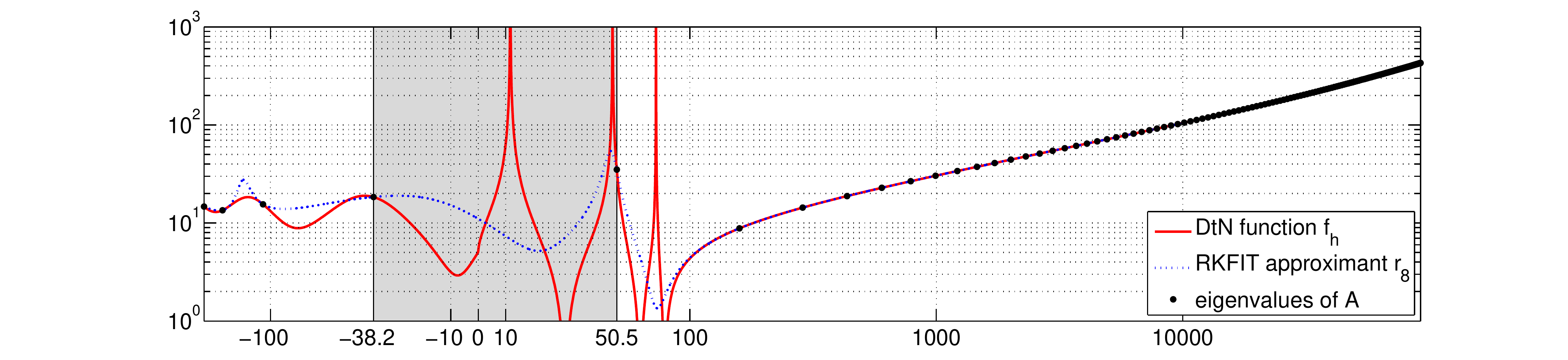}
\end{minipage}
\caption{A waveguide with varying wave number in the $x$-direction (piecewise constant over the first 150 grid points and the remaining grid points until infinity). The top row shows the amplitude and phase of the solution, with the position of the coefficient jump highlighted by vertical dashed line. The bottom shows a plot of the exact DtN function $f_h$ (solid red line) over the spectral interval of the indefinite matrix $A$. The plot is  logarithmic on both axes, with the $x$-axis showing a negative and positive part of the real axis, glued together by the gray linear part in between.   The RKFIT approximant of degree $n=8$ (dotted blue curve) exhibits spectral adaptation to some of $A$'s eigenvalues (black dots).} 
\label{fig:waveguide}
\end{figure}

Our RKFIT approach is also applicable when $A$ is non-Hermitian, which may result from absorbing boundary conditions in the transversal plane.  We demonstrate in several experiments that the RKFIT-FD grids are exponentially accurate as an approximation to the full FD scheme, with only  a small number of grid points required for practical accuracy. As a result of spectral adaptation effects, the Nyquist limit of two grid points per wavelength  does not fully apply to RKFIT-FD grids.  For the problem in  Figure~\ref{fig:waveguide}, for example, we computed an RKFIT-FD grid of only $n=8$  points which accurately (to about six digits of  relative accuracy) mimics the response of the full variable-coefficient waveguide discretized by $300$ grid points in the $x$-direction. This is a significant compression of the full grid.

\rev{The} rest of this paper is structured as  follows: in section~\ref{sec:fdcf} we derive analytic expressions of  DtN maps for constant- and variable-coefficient media. We relate the optimization of these \rev{DtN maps  to}  approximation problems. Section~\ref{sec:fdrk} establishes a new connection between rational Krylov spaces and FD grids. In section~\ref{sec:rkfit} we tailor the  RKFIT  algorithm to our specific application. 
Sections~\ref{sec:conv1} and \ref{sec:conv2}  study the convergence behaviour of the algorithm. In section~\ref{sec:disc} we discuss the numerical results and compare  them to the Nyquist limit and other (spectral) discretization schemes. In the appendix we give a rational approximation interpretation of the Nyquist limit and explain why  this limit is not necessarily strict for RKFIT-FD grids.



\section{From DtN maps to continued fractions and FD grids}\label{sec:fdcf}
There is an \rev{intimate} connection between FD grids and rational functions.  
To see this, let us first consider the scalar ODE $u''(x) = \lambda u(x)$ on $x\geq 0$ and its  FD discretization
\begin{equation}\label{eq:grid1}
h^{-1}\left[ h^{-1}(u_{j+1}-u_j) - h^{-1}(u_j-u_{j-1})\right] = \lambda u_j, \quad j = 1,2,\ldots,
\end{equation}
where $\lambda$ and $u_0$ are given constants  and  we demand that $u_n$ remains bounded as $n\to \infty$.  This linear recurrence  is a scalar version of \eqref{eq:MRel2v} with $c\equiv 0$. It can easily be solved  by computing the roots of the  characteristic polynomial $p(t) = (t^2 - (2+h^2\lambda) t + 1)/h^2$ and choosing the  solution
$u_j =  ( 1 + {h^2\lambda}/{2} - h\sqrt{\lambda +h^2\lambda^2/4} 
 )^j \cdot u_0$. 
Indeed this is the only solution that decays for  $\lambda > 0$. Moreover, this solution is bounded under the condition\footnote{This is an interesting condition in the indefinite Helmholtz case, where the role of $\lambda$ is  played by the eigenvalues of the shifted Laplacian $-\nabla^2 - k^2$ and  $k$ is the wave number. Because we require $\lambda \geq -4/h^2$, we have a condition $k^2 \leq 4/h^2$ on the wave number, which is equivalent to  $kh \leq 2$.  The solution of the Helmholtz equation in a homogeneous medium has wave length $\ell = 2\pi/k$. Hence the number of FD grid points per wavelength, $n = \ell/h$, must satisfy $n = \ell/h = 2\pi/(kh) \geq \pi$ in order to approximate a bounded oscillatory solution.}
  $\lambda \geq -4/h^2$ and unbounded for  $\lambda < -4/h^2$.

We can use the explicit solution $\{u_j\}$ to extract interesting  information about the problem. For example, from the FD relation $
2h^{-1} \left[ h^{-1}(u_1 - u_0) + b \right] = \lambda u_0$, the scalar version of \eqref{eq:MRel1v}, 
we obtain an approximation $b$ to the Neumann boundary data $-u'(x=0)$ for the continuous analogue of the FD scheme. Eliminating $u_1$ using the above formula, we can directly relate  $u_0$ and $b$ via
$b = \sqrt{ \lambda + {h^2\lambda^2}/{4}}\, u_0 =: f_h(\lambda) u_0$. 
We refer to $f_h$ as the \emph{DtN function} or  \emph{discrete impedance function}.  
By letting $h\to \infty$ we recover the DtN relation $b = \sqrt{\lambda} u_0 =: f(\lambda) u_0$ and indeed $b = -u'(0)$ for the continuous solution $u(x) = \exp(-x\sqrt{\lambda}) u_0$. 


Now let us  turn  to the variable-coefficient problem \eqref{eq:MRelv} in scalar form:
\begin{subequations}\label{eq:vcdtn0}
\begin{eqnarray}
 2h^{-1}\left[ h^{-1}({u}_1 - {u}_0) + b \right] &=&  (\lambda+c_0) {u}_0  \\
h^{-1}\left[h^{-1}
 ({u}_{j+1} - {u}_j) - h^{-1}(u_{j} - u_{j-1})
        \right] &=&   (\lambda+c_j) {u}_j ,
        \quad   j = 1,2,\ldots. 
\end{eqnarray}
\end{subequations}
By eliminating the grid points with indices $j>L$ (where $c_j=0$) we find the DtN relation $b/u_0 = f_h(\lambda)$ in continued fraction form  
\rev{
\begin{small}
\begin{equation}\label{eq:vcdtn}
 f_h(\lambda) = 
    \frac{h (\lambda+c_0)}{2} + \cfrac{1}
                      {h + \cfrac{1}
                             {  h (\lambda+c_1) + \cfrac{1} {h + \cdots + \cfrac{1}
                                                                   {  h (\lambda+c_L) + \cfrac{1}{h + \cfrac{1}{\cfrac{h\lambda}{2} + \sqrt{\lambda + \frac{h^2 \lambda^2}{4}}}}}}}}\,.
\end{equation}
\end{small}
}

\noindent In view of the original vector-valued problem \eqref{eq:MRelv}, the role of $\lambda$ is played by the eigenvalues of the matrix~$A$. 
When employing a rational approximant $r_n\approx f_h$ it hence seems reasonable to be accurate on the \emph{spectral region of $A$.}  
For example, if $A$ is diagonalizable as $A = X \diag (\lambda_1,\lambda_2,\ldots, \lambda_N) X^{-1}$, we have 
$\| f_h(A) - r_n(A) \|_2 \leq \|X\|_2 \|X^{-1}\|_2 \max_{1\leq j \leq N} |f_h(\lambda_j) - r_n(\lambda_j) |$. 
Hence if the condition number $\kappa(X) = \|X\|_2 \|X^{-1}\|_2$ is moderate, we can bound the accuracy of $r_n(A)$  using a scalar approximation problem on the eigenvalues $\lambda_j$.   
The rational approximant $r_n$ can be viewed as a reduced order model of  \eqref{eq:vcdtn0} where {the} spectral parameter $\lambda$ of the transversal operator is an equivalent of   the temporal (Laplace) frequency in linear time invariant dynamical systems (see, e.g., \cite{beattie2017model}).

%
%

\section{From FD grids to rational Krylov spaces}\label{sec:fdrk}
The crucial observation for optimizing the rational approximant $r_n \approx f_h$ of a DtN function, or equivalently its associated FD grid, is that the grid steps do not need to be equispaced, and not even real-valued.  
Consider the FD scheme 
\begin{subequations}\label{eq:Rel}
\begin{eqnarray}
{\widehat h_0}^{-1}\left[({u}_1 - {u}_0) + b \right] &=&  \lambda {u}_0 \label{eq:Rel1} \\
{\widehat h_j}^{-1}\left[
h_{j+1}^{-1} ({u}_{j+1} - {u}_j)  - h_{j}^{-1}({u}_{j} - {u}_{j-1})
        \right] &=&   \lambda {u}_j ,
        \quad   j = 1,\ldots,n-1 \label{eq:Rel2}
\end{eqnarray}
\end{subequations}
with arbitrary complex-valued primal and dual grid steps $h_j$ and $\widehat h_{j-1}$ ($j=1,2,\ldots,n$), respectively. 
The continued fraction form of the associated DtN maps, derived in exactly the same manner as for the case of constant $h$ in section~\ref{sec:fdcf}, is
\rev{
\begin{equation}\label{eq:contfrach}
  r_n(\lambda)   = 
    \widehat h_0 \lambda  + \cfrac{1}
         {h_1 + \cfrac{1}
                {\widehat h_1  \lambda + \cfrac{1} 
                {h_2 + \cdots + \cfrac{1}{\widehat h_{n-1} \lambda + \cfrac{1}{h_n}} }}}\,.
\end{equation}}
This is a rational function of type $(n,n-1)$, i.e., a quotient $p_n/q_{n-1}$ of polynomials of degree $n$ and $n-1$, respectively. By choosing the free grid steps we can optimize it for our purposes. In particular, we can tune \eqref{eq:Rel} so that it implements a rational approximation to any DtN map, even if the associated analytic DtN function $f_h$ is complicated. To this end, we need a robust method for computing such rational approximants and a numerical conversion into continued fraction form.

The vector form of \eqref{eq:Rel}  is 
\begin{subequations}\label{eq:Relv}
\begin{eqnarray}
{\widehat h_0}^{-1}\left[ h_1^{-1}(\mathbf{u}_1 - \mathbf{u}_0) + \mathbf b \right] &=&  A \mathbf{u}_0  \label{eq:Rel1v} \\
{\widehat h_j}^{-1}\left[
h_{j+1}^{-1} (\mathbf{u}_{j+1} - \mathbf{u}_j)  - h_{j}^{-1}(\mathbf{u}_{j} - \mathbf{u}_{j-1})
        \right] &=&   A \mathbf{u}_j ,
        \quad   j = 1,\ldots,n-1. \label{eq:Rel2v}
\end{eqnarray}
\end{subequations}
Again, $\mathbf{b} = r_n(A)\mathbf{u}_0$ with a rational function $r_n=p_n/q_{n-1}$ whose continued fraction form \eqref{eq:contfrach} involves the grid steps $h_j$ and $\widehat h_{j-1}$. 
The vectors $\mathbf{u}_j$ and $\mathbf{b}= r_n(A) \mathbf{u}_0$ satisfy a \emph{rational Krylov decomposition} 
\begin{equation}\label{eq:arnU}
        A U_{n+1}\underline{\widetilde K_n} = U_{n+1} \underline{\widetilde H_n},
\end{equation}
where $U_{n+1} = [ \, r_n(A)\mathbf{u}_0\, |\, \mathbf{u}_{0}\, |\, \mathbf{u}_{1}\, |\, \cdots\, |\, \mathbf{u}_{n-1}\, ]\in\mathbb{C}^{N\times (n+1)}$ and 
 $\underline{\widetilde K_n}, \underline{\widetilde H_n}\in\mathbb{C}^{(n+1)\times n}$ are  
\begin{equation}\label{eq:pic2}
        \hspace*{-1mm}\underline{\widetilde K_n} = 
        \begin{bmatrix}
         0 &  \\
        \widehat h_0 &   &   &      \\
          & \widehat h_1 &   &      \\
          &   & \ddots &    \\
          &   &   & \widehat h_{n-1}   \\
        \end{bmatrix}, \ \underline{\widetilde H_n} = 
        \begin{bmatrix}
        1 & \\
        { -h_1^{-1}} & { h_1^{-1}} &   &     \\
        { h_1^{-1}} & { -h_1^{-1} - h_2^{-1}} & \ddots &      \\
          & \ddots & \ddots &   h_{n-1}^{-1}  \\
          &   &  h_{n-1}^{-1} & -h_{n-1}^{-1} - h_n^{-1}   \\
        \end{bmatrix} .
\end{equation}
The entries in  $(\underline{\widetilde H_n},\underline{\widetilde K_n})$ encode the recursion coefficients in \eqref{eq:Rel}   and the columns of $U_{n+1}$ all correspond to rational functions in $A$ multiplied by the vector~$\mathbf{u}_0$. \rev{More precisely, 
\[
\mathrm{colspan}(U_{n+1}) = q_{n-1}(A)^{-1} \mathrm{span} \{ \mathbf{u}_0, A\mathbf{u}_0, \ldots, A^n \mathbf{u}_0 \}
\]
for some denominator polynomial $q_{n-1}$ of degree at most~$n-1$ and with no roots at any of $A$'s eigenvalues.
Such a space is also known as a \emph{rational Krylov space}~\cite{Ruh98}.}
In the next section we will show how to generate decompositions of the form \eqref{eq:arnU} numerically and how to interpret them as FD grids.

\section{The RKFIT approach}\label{sec:rkfit}

Assume that $F,A \in \mathbb{C}^{N\times N}$ are given matrices  and $\mathbf{v}\in\mathbb{C}^N$ with $\| \mathbf{v} \|_2=1$. Our aim is to find a rational approximant $r_n(A)\mathbf{v}$ such that 
\begin{equation}\label{eq:min}
        \| F\mathbf{v} - r_n(A)\mathbf{v} \|_2 \to \min.
\end{equation}
For the purpose of this paper, $F$ is the linear DtN map  and the sought rational function $r_n = p_{n}/q_{n-1}$ is of type $(n,n-1)$. 
As \eqref{eq:min} is a nonconvex optimization problem it may have many solutions, exactly one solution, or no solution at all. However, this difficulty has not prevented the development of algorithms for the (approximate) solution of \eqref{eq:min}; see \cite{BG15} for a discussion of various algorithms. The RKFIT algorithm   \cite{BG14,BG15} is particularly  suited for this task and in this section we shall briefly review it and adapt it to our application.

\subsection{Search and target spaces}
Given a set of \emph{poles} $\xi_1,\xi_2,\ldots,\xi_{n-1}\in\mathbb{C}$ and an associated nodal polynomial $q_{n-1}(\lambda) = \prod_{j=1}^{n-1} (\lambda - \xi_j)$, 
RKFIT makes use of two spaces, namely an $n$-dimensional \emph{search space} $\mathcal{V}_{n}$ defined as
$\mathcal{V}_{n} := q_{n-1}(A)^{-1} \mathcal{K}_{n}(A,\mathbf{v})$,
and an $(n+1)$-dimensional \emph{target space} $\mathcal{W}_{n+1}$ defined as
$\mathcal{W}_{n+1} := q_{n-1}(A)^{-1} \mathcal{K}_{n+1}(A,\mathbf{v})$. 
Here, $\mathcal{K}_{j}(A,\mathbf{v}) = \mathrm{span} 
\{ \mathbf{v},A\mathbf{v},$ $\ldots,A^{j-1}\mathbf{v} \}$ is the standard (polynomial) Krylov space \rev{of dimension $j$ for the matrix~$A$ and starting vector $\mathbf{v}$.} Let $V_{n}\in\mathbb{C}^{N\times n}$ and $W_{n+1}\in\mathbb{C}^{N\times (n+1)}$ be orthonormal bases for $\mathcal{V}_{n}$ and $\mathcal{W}_{n+1}$, respectively. 
 
The space $\mathcal{V}_{n}$ is a rational Krylov space with starting vector $\mathbf v$ and the poles $\xi_1,\ldots,\xi_{n-1}$, i.e., a linear space of type $(n-1,n-1)$ rational functions $(p_j/q_{n-1})(A)\mathbf{v}$, all sharing the same denominator $q_{n-1}$. 
As a consequence, we can arrange the columns of $V_{n}$ such that $V_{n} \mathbf{e}_1 = \mathbf{v}$ and a rational Krylov decomposition
\begin{equation}\label{eq:arnoldi}
        A V_{n} \underline{K_{n-1}} = V_{n} \underline{H_{n-1}}
\end{equation}
is satisfied. \rev{The existence of such a decomposition under the assumption that $\mathcal{V}_{n}$ is a rational Krylov space is shown in \cite[Thm.~2.5]{BG14}. For a given sequence of poles $\xi_1,\ldots,\xi_{n-1}$,  decompositions of this form are computed by Ruhe's rational Krylov sequence (RKS) algorithm~\cite[Section~2]{Ruh98} and its variant described in~\cite[Algorithm~2.1]{BG17parallel}.}  Here, $(\underline{H_{n-1}},\underline{K_{n-1}})$ is an unreduced upper Hessenberg pair of size $n\times (n-1)$, i.e., both $\underline{H_{n-1}}$ and $\underline{K_{n-1}}$ are upper Hessenberg matrices which do not share a common zero element on the subdiagonal. 
The following result, established in \cite[Thm.~2.5]{BG14}, relates the generalized eigenvalues of the lower $(n-1)\times (n-1)$ submatrices of $(\underline{H_{n-1}},\underline{K_{n-1}})$, the poles of the rational Krylov space, and its starting vector.

\smallskip

\begin{theorem}\label{thm:pole}
The generalized eigenvalues of the lower $(n-1)\times (n-1)$ submatrices of $(\underline{H_{n-1}},\underline{K_{n-1}})$ of \eqref{eq:arnoldi} are the poles $\xi_1,\ldots,\xi_{n-1}$ of the rational Krylov space $\mathcal{V}_{n}$ with starting vector $\mathbf v$.

Conversely, let a decomposition $A\widehat V_n  \underline{\widehat K_{n-1}} = \widehat V_{n} \underline{\widehat H_{n-1}}$  with $\widehat V_n\in\mathbb{C}^{N\times n}$ of full column rank and an unreduced upper Hessenberg pair $(\underline{\widehat H_{n-1}},\underline{\widehat K_{n-1}})$ be given. Assume further that none of the generalized eigenvalues $\widehat \xi_j$ of the lower $(n-1)\times (n-1)$ submatrices of $(\underline{\widehat H_{n-1}},\underline{\widehat K_{n-1}})$ coincides with an eigenvalue of $A$. Then the columns of $\widehat V_n$ form a basis for a rational Krylov space with starting vector $\widehat V_n\mathbf{e}_1$ and poles $\widehat \xi_j$. 
\end{theorem}

\subsection{Pole relocation and projection step}

The main component of RKFIT is a pole relocation step  based on Theorem~\ref{thm:pole}.
Assume that a guess for the denominator polynomial $q_{n-1}$ is available and orthonormal bases $V_{n}$ and $W_{n+1}$ for the spaces $\mathcal{V}_{n}$ and $\mathcal{W}_{n+1}$ have been computed. Then we can identify a vector $\mathbf{\widehat v}\in\mathcal{V}_n$, $\|\mathbf{\widehat v}\|_2 = 1$, such that  
$F\mathbf{\widehat v}$ is best approximated by some vector in $\mathcal{W}_{n+1}$. More precisely, we can find a coefficient vector $\mathbf{c}_n\in\mathbb{C}^n$, $\|\mathbf{c}_n\|_2=1$, such that
$\| (I_N - W_{n+1}W_{n+1}^*)F V_n \mathbf{c}_n \|_2 \to \min$.
The vector $\mathbf{c}_n$ is given as a right singular vector of $(I_N - W_{n+1}W_{n+1}^*)F V_n$ corresponding to a smallest singular value. 

Assume that a ``sufficiently good'' denominator $q_{n-1}$ of $r_n = p_n/q_{n-1}$ has been found. Then the problem of finding the numerator $p_{n}$ such that $\| F\mathbf v - r_n(A)\mathbf v \|_2$ is minimal becomes a linear one. Indeed, the vector $r_n(A)\mathbf v := W_{n+1} W_{n+1}^* F \mathbf v$ corresponds to the orthogonal projection of $F\mathbf v$ onto $\mathcal{W}_{n+1}$ and its representation in the rational Krylov basis $W_{n+1}$ is 
\begin{equation}\label{eq:rnAv}
    r_n(A)\mathbf v = W_{n+1} \mathbf{c}_{n+1}, \quad  \text{where} \ \ \ \mathbf{c}_{n+1} := W_{n+1}^* F \mathbf v.
\end{equation}
The pseudocode for a single RKFIT iteration is given in Algorithm~\ref{alg:rkfit}. A MATLAB implementation is contained in the Rational Krylov Toolbox  \cite{BG14b} which is available online at \url{http://rktoolbox.org}.

\begin{algorithm}[t!]
  \caption{One RKFIT iteration for superdiagonal approximants.}
  \label{alg:rkfit}
  \begin{algorithmic}[1]
    \REQUIRE Matrices $A, F\in\mathbb{C}^{N\times N}$, nonzero $\mathbf v\in\mathbb{C}^N$, and initial poles $\xi_1,\xi_2,\ldots,\xi_{n-1}\in\mathbb{C}\setminus\Lambda(A)$ (in the first iteration it is recommended to initialize all poles at $\infty$).\\[1.5mm]
    \ENSURE Improved poles $\widehat \xi_1,\widehat \xi_2,\ldots,\widehat \xi_{n-1}$.
    \vskip 5 pt
    \STATE Compute a rational Krylov decomposition $AW_{n+1} \underline{K_n} =  W_{n+1} \underline{H_n}$
    with $W_{n+1}\mathbf e_1=\mathbf v/\|\mathbf v\|_2$ and poles $\xi_1,\xi_2,\ldots,\xi_{n-1},\infty$. \\[1.5mm]
    \STATE Define $V_n = W_{n+1} [ \, I_n\, |\, \vnull\,  ]^T$.\\[1.5mm]
    \STATE Compute a right singular vector $\mathbf c_n\in\mathbb{C}^{n}$ of $(I - W_{n+1} W_{n+1}^*) F V_{n}$
    corresponding to a smallest singular value.\\[1.5mm]
    \STATE {Form $A\widehat V_{n}\underline{\widehat H_{n-1}} = \widehat V_{n}\underline{\widehat K_{n-1}}$
    spanning $\mathcal{R}(V_{n})$ with $\widehat V_n \mathbf e_1 = V_{n}\mathbf c_n$.}\\[1.5mm]
    \STATE Compute $\widehat \xi_1, \widehat \xi_2,\ldots, \widehat \xi_{n-1}$ as the generalized eigenvalues of the lower $(n-1)\times (n-1)$ part of 
    $(\underline{\widehat H_{n-1}},\underline{\widehat K_{n-1}})$.
  \end{algorithmic}
\end{algorithm}

\subsection{Conversion to continued fraction form}\label{subsec:trafo}

Similarly to what we did in \eqref{eq:arnoldi}, we can arrange the columns of $W_{n+1}$ so that 
$W_{n+1} \mathbf{e}_1 = \mathbf{v}$ and a rational Krylov decomposition
\begin{equation}\label{eq:arnoldiw}
        A W_{n+1} \underline{K_{n}} = W_{n+1} \underline{H_{n}}
\end{equation}
is satisfied, where $(\underline{H_{n}},\underline{K_{n}})$ is an unreduced upper Hessenberg pair of size $(n+1)\times n$. Indeed, we have $\mathcal{V}_n \subset \mathcal{W}_{n+1}$ and $\mathcal{W}_{n+1}$ is a rational Krylov space with starting vector~$\mathbf v$, finite poles $\xi_1,\ldots,\xi_{n-1}$, and a formal additional ``pole'' at $\infty$.

Our aim is to transform the decomposition~\eqref{eq:arnoldiw} so that it can be identified with \eqref{eq:arnU} when $\mathbf{u}_0 = \mathbf{v}$. This transformation should not alter the space $\mathcal{W}_{n+1}$ but merely transform the basis $W_{n+1}$ into the continued fraction basis $U_{n+1}$ and the pair $(\underline{H_n},\underline{K_n})$ into the tridiagonal-and-diagonal form of \eqref{eq:pic2}. 

First we transform~\eqref{eq:arnoldiw} so that $r_n(A)\mathbf{v}$ defined in \eqref{eq:rnAv} becomes the first vector in the rational Krylov basis, and  $\mathbf{v}$ the second. To this end, we define the transformation matrix 
$X = [ \, \mathbf{c}_{n+1}\, |\, \mathbf{e}_1\, |\, \mathbf{x}_3\, |\, \cdots\, \, \mathbf{x}_{n+1} ] \in\mathbb{C}^{(n+1)\times (n+1)}$ 
with the columns $\mathbf{x}_3,\ldots,\mathbf{x}_{n+1}$ chosen freely but so that $X$ is invertible,  and rewrite \eqref{eq:arnoldiw} by inserting $X X^{-1}$:
\begin{equation}\label{eq:arnoldi2}
        A  W_{n+1}^{(0)} \underline{K_n^{(0)}} = W_{n+1}^{(0)} \underline{H_n^{(0)}},
\end{equation}
where $W_{n+1}^{(0)} = W_{n+1} X$, $\underline{K_n^{(0)}} = X^{-1} \underline{K_n}$ and $\underline{H_n^{(0)}} = X^{-1} \underline{H_n}$.
By construction, the transformed rational Krylov basis $W_{n+1}^{(0)}$ is of the form
$W_{n+1}^{(0)} = \big[ \, r_n(A)\mathbf{v} \, |\, \mathbf{v} \, |  * \, | \, \cdots \, | \, *\,   \big]\in\mathbb{C}^{N\times (n+1)}$. 
The transformation to \eqref{eq:arnoldi2} has potentially destroyed the upper Hessenberg structure of the decomposition and $(\underline{H_n^{(0)}},\underline{K_n^{(0)}})$ generally is a dense $(n+1)\times n$ matrix pair. Here is a pictorial view of decomposition \eqref{eq:arnoldi2} for the case $n=4$:
\begin{equation}\label{eq:pic1}
        A W_{n+1}^{(0)} 
        \begin{bmatrix}
        * & * & * & *  \\
        * & * & * & * \\
        * & * & * & *  \\
        * & * & * & * \\
        * & * & * & *  \\
        \end{bmatrix}
        =
        W_{n+1}^{(0)} 
        \begin{bmatrix}
        * & * & * & *  \\
        * & * & * & *  \\
        * & * & * & *  \\
        * & * & * & *  \\
        * & * & * & *  \\
        \end{bmatrix}.
\end{equation}

We now transform  $(\underline{H_n^{(0)}},\underline{K_n^{(0)}})$ into tridiagonal-and-diagonal form by successive right and left multiplication, giving rise to pairs  $(\underline{H_n^{(j)}},\underline{K_n^{(j)}})$ ($j=1,2,\ldots,5$) all corresponding to the same rational Krylov space $\mathcal{W}_{n+1}$ and all without the two leading vectors in $W_{n+1}^{(0)}$ being altered. More precisely, the  allowed transformations are:
\begin{itemize}
\item \underline{right-multiplication of the pair} by any invertible matrix $R\in\mathbb{C}^{n\times n}$,
\item \underline{left-multiplication of the pair} by an invertible matrix $L\in\mathbb{C}^{(n+1)\times (n+1)}$, the first two columns of which are $[\, \mathbf{e}_1\, |\, \mathbf{e}_2\, ]$. This ensures that inserting $L^{-1}L$ into the decomposition will not alter the leading two vectors $[\, r_n(A)\mathbf{v}\, |\, \mathbf{v} \, ]$ in the rational Krylov basis.
\end{itemize}

\noindent Here are the transformations we perform:
\begin{enumerate}
\item 
We right-multiply the pair $(\underline{H_n^{(0)}},\underline{K_n^{(0)}})$ by the inverse of the lower $n\times n$ part of 
$\underline{K_n^{(0)}}$, giving rise to  $(\underline{H_n^{(1)}},\underline{K_n^{(1)}})$ (we now only show a pictorial view of the transformed pairs):
\[
        A W_{n+1}^{(1)} \begin{bmatrix}
        {\small 0} & * & * & * \\
        {\small 1} &   &   &       \\
          & {\small 1} &   &       \\
          &   & {\small 1} &       \\
      &   &   & {\small 1}     \\
        \end{bmatrix} 
        = W_{n+1}^{(1)} 
        \begin{bmatrix}
        * & * & * & *   \\
        * & * & * & *   \\
        * & * & * & *   \\
        * & * & * & *   \\
        * & * & * & *   \\
        \end{bmatrix}.
\]
The Krylov basis matrix $W_{n+1}^{(1)} = W_{n+1}^{(0)} = [ \, r_n(A)\mathbf{v}\, | \, \mathbf{v} \, |\, * \, | \, \cdots \, | \, *\, ]$ has  not changed. 
The $(1,1)$ element of the transformed matrix $\underline{K_n^{(1)}} = [ k_{ij}^{(1)} ]$ is automatically zero because the  decomposition states that  the linear combination $k_{11}^{(1)} Ar_n(A)\mathbf{v} + k_{21}^{(1)} \mathbf{v}$ is in the column span of $W_{n+1}^{(1)}$, a space of type $(n,n-1)$ rational functions. This linear combination is a type $(n+1,n-1)$ rational function unless $k_{11}=0$.\\[-1mm]
\item We left-multiply the pairs to zero the first row of $\underline{K_n^{(1)}}$ completely. This can be done by adding multiples of the 3rd, 4th,\,\dots,\,$(n+1)$th row to the first. As a result we obtain
\begin{equation}\label{eq:h2}
        A W_{n+1}^{(2)}  \begin{bmatrix}
        {\small 0}  &   &   &      \\
        {\small 1} &   &   &      \\
          & {\small 1} &   &      \\
          &   & {\small 1} &      \\
      &   &   & {\small 1}    \\
        \end{bmatrix} = 
        W_{n+1}^{(2)} \begin{bmatrix}
        * & * & * & *  \\
        * & * & * & *  \\
        * & * & * & *  \\
        * & * & * & *  \\
        * & * & * & *  \\
        \end{bmatrix}.
\end{equation}
This left-multiplication does not affect the leading two columns of the Krylov basis, hence $W_{n+1}^{(2)}$ is still of the form $W_{n+1}^{(2)} = [ \, r_n(A)\mathbf{v}\, | \, \mathbf{v} \, |\, * \, | \, \cdots \, | \, *\, ]$.\\[-1mm]
\item We  right-multiply the pair to zero all elements in the first row of $\underline{H_n^{(2)}}$ except the $(1,1)$ entry, which we can assume to be nonzero (see Remark~\ref{rem:h2}). This can be done by adding multiples of the first column to the 2nd, 3rd,\,\dots ,\,$n$th column.   As a result we have 
\[
        A W_{n+1}^{(3)} \begin{bmatrix}
         {\small 0} &   &   &    \\
        {\small 1} & * & * & *  \\
          & {\small 1} &   &      \\
          &   & {\small 1} &      \\
      &   &   & {\small 1}    \\
        \end{bmatrix} = 
        W_{n+1}^{(3)} \begin{bmatrix}
        * &   &   &      \\
        * & * & * & *  \\
        * & * & * & *  \\
        * & * & * & *  \\
        * & * & * & *  \\
        \end{bmatrix}.
\]
Again, this right-multiplication has not affected $W_{n+1}^{(3)}=W_{n+1}^{(2)}$. \\[-1mm]
\item With a further left-multiplication, adding multiples of the 3rd, 4th,\dots,\,\mbox{$(n+1)$st} row to the second row, we can zero all the entries in the second row of $\underline{K_n^{(3)}}$, except the entry in the $(2,1)$ position:
 \[
        A W_{n+1}^{(4)}  \begin{bmatrix}
         {\small 0} &   &   &      \\
        {\small 1} &   &   &      \\
          & {\small 1} &   &      \\
          &   & {\small 1} &      \\
      &   &   & {\small 1}    \\
        \end{bmatrix} = W_{n+1}^{(4)} 
        \begin{bmatrix}
        * &   &   &    \\
        * & * & * & *  \\
        * & * & * & *  \\
        * & * & * & *  \\
        * & * & * & *  \\
        \end{bmatrix}.
\]
Note that $\underline{H_n^{(4)}}$ still has zero entries in its first row. Also, $W_{n+1}^{(4)}$ is still of the form $W_{n+1}^{(4)} = [ \, r_n(A)\mathbf{v}\, | \, \mathbf{v} \, |\, * \, | \, \cdots \, | \, *\, ]$.\\[-1mm]
\item We apply the two-sided Lanczos algorithm with the lower $n\times n$ part of $\underline{H_n^{(4)}}$, using $\mathbf{e}_1$ as the left and right starting vector. This produces biorthogonal matrices $Z_L,Z_R\in\mathbb{C}^{n\times n}$, $Z_L^H Z_R=I_n$. Left-multiplying the decomposition with $\mathrm{blkdiag}(1,Z_L^H)$ and right-multiplication with $Z_R$ results in the demanded structure: 
 \begin{equation} \label{eq:strucwanted}
        A W_{n+1}^{(5)} \begin{bmatrix}
         {\small 0} &   &   &    \\
        {\small 1} &   &   &     \\
          & {\small 1} &   &     \\
          &   & {\small 1} &     \\
      &   &   & {\small 1}   \\
        \end{bmatrix} = W_{n+1}^{(5)}
        \begin{bmatrix}
        * &   &   &     \\
        * & * &   &    \\
        * & * & * &    \\
          & * & * & *   \\
          &   & * & *  \\
        \end{bmatrix}.
\end{equation}
\item Finally, let the nonzero entries of $\underline{H_n^{(5)}}$ be denoted by $\eta_{i,j}$ ($1\leq j\leq n$, $j\leq i\leq j+2$), then we aim to scale these entries so that they are matched with those of the matrix  $\underline{\widetilde H_n}$ in \eqref{eq:pic2}. This can be achieved by left multiplication of the pair with $L=\diag(1,1,\ell_3,\ldots,\ell_{n+1})\in\mathbb{C}^{(n+1)\times (n+1)}$ and right multiplication with $R=\diag(\rho_1,\rho_2,\ldots,\rho_n)\in\mathbb{C}^{n\times n}$. The diagonal entries of $L$ and $R$ are found by equating $\underline{\widetilde H_n}$ in \eqref{eq:pic2} and $L \underline{H_n^{(5)}} R$, starting from the $(1,1)$ entry and going down columnwise. We obtain
$r_1 = 1/\eta_{1,1}$,\ $h_1 = -1/(\eta_{2,1} \rho_1)$,\ 
$\ell_3 = 1/(\eta_{3,1} h_1 \rho_1)$,
and for $j=2,3,\ldots$
$r_j = 1/( \ell_j \eta_{j,j} h_{j-1})$,\ 
$h_j = -1/(1/h_{j-1} + \ell_{j+1} \eta_{j+1,j} \rho_j)$,\ 
$\ell_{j+2} = 1/(\eta_{j+2,j} h_j \rho_j)$.
The diagonal entries of $\underline{\widetilde K_n}$ in \eqref{eq:pic2} satisfy
$\widehat h_{j-1} = \ell_{j+1} \rho_j$,   \ $j = 1,\ldots,n$,
and thus the pair has been transformed exactly into the form \eqref{eq:pic2}.
\end{enumerate}

\smallskip

The above six-step procedure converts the RKFIT approximant  $r_n$ into continued fraction form  and hence allows its interpretation as an FD scheme. This scheme is referred to as an RKFIT-FD grid. Note that all transformations only act on small matrices of size $(n+1)\times n$ and the computation of the tall skinny matrices $W_{n+1}^{(j)}$ is not required if one   only needs the continued fraction parameters. 
 We have extended the Rational Krylov Toolbox by the \texttt{contfrac} method, which implements the conversion of an RKFUN, the fundamental data type to represent and work with rational functions~$r_n$, into continued fraction form following the above transformations. Numerically, these transformations may be ill conditioned and the use of multiple precision arithmetic is recommended. The toolbox supports   MATLAB's Variable Precision Arithmetic  and  the Advanpix Multiprecision Toolbox   \cite{advanpix}.

\begin{remark}\label{rem:h2}
In Step~3 we have assumed that the $(1,1)$ element of $\underline{H_n^{(2)}}$ is nonzero. This assumption is always satisfied: assuming to the contrary that the $(1,1)$ element of $\underline{H_n^{(2)}}$ vanishes, the first column of \eqref{eq:h2} reads $A\mathbf{v} = W_{n+2}^{(2)} [0,*,\ldots,*]^T$. This is a contradiction as the left-hand side of this equation is a superdiagonal rational function in $A$ times $\mathbf{v}$, whereas the trailing $n$ columns of $W_{n+1}^{(2)}$ can be taken to be a basis for $\mathcal{V}_n\subset \mathcal{W}_{n+1}$, which only contains diagonal (and subdiagonal) rational functions in $A$ times $\mathbf{v}$ (provided that all poles $\xi_1,\ldots,\xi_{n-1}$ are finite). 
\end{remark}

\begin{remark}\label{rem:h5}
In Step~5 we have assumed that the lower $n\times n$ part of $\underline{H_n^{(4)}}$ can be tridiagonalized by the two-sided Lanczos algorithm. While this conversion can potentially fail, we conjecture that if $r_n$ admits a continued fraction form \eqref{eq:contfrach} then such an unlucky breakdown cannot occur. (The conditions for the rational function $(r_n(\lambda) - \widehat h_0 \lambda)$ to posses this so-called Stieltjes continued fraction form \cite{Stieltjes1894} are reviewed in  \cite{Holtz}; see Theorem~1.39 therein.) Even if our  conjecture was false, the starting vector $\mathbf{v}$  will typically be chosen at random in our application. So  if an unlucky breakdown occurs, trying again with another vector $\mathbf{v}$ would easily solve the problem. We have not encountered any unlucky breakdowns in our experiments.
\end{remark}

\section{Numerical tests: constant-coefficient case}\label{sec:conv1}

The nonlinear rational least squares problem \eqref{eq:min} is nonconvex and there is no guarantee that a minimizing solution exists, nor that such a solution would be unique. 
As a consequence of these theoretical difficulties and due to the nonlinear nature of RKFIT's pole relocation procedure, a comprehensive convergence analysis seems currently intractable. (An exception is \cite[Corollary~3.2]{BG15}, which states that in exact arithmetic RKFIT converges within a single iteration if $F$ itself is a rational matrix function of appropriate type.) However, for some special cases we can compare the RKFIT approximants to analytically constructed near-best approximants. Here we provide such comparisons to the compound Zolotarev approach in \cite{DGK16}  and the   approximants studied by Newman and Vjacheslavov \cite[Section~4]{PP87}. 

Throughout this section we assume that $A$ is Hermitian with eigenvalues $\lambda_1\leq \lambda_2\leq \cdots \leq \lambda_N$. 
In our discussion of available convergence bounds we will usually focus on the function $f(\lambda)=\sqrt{\lambda}$, however, as has been argued in \cite[Section~5.1]{DGK16}, it is possible to obtain similar bounds for the discrete impedance function $f_h(\lambda) = \sqrt{\lambda + (h\lambda/2)^2}$. Some of our numerical experiments will be for the latter function, illustrating that the convergence behavior is indeed similar to that for the former.


%
%

\subsection{Two-interval approximation with coarse spectrum}
Our first test concerns the approximation of $F = f_h(A)$, $f_h(\lambda) = \sqrt{\lambda+(h\lambda/2)^2}$, where $A$ is a nonsingular indefinite Hermitian matrix with  relatively large gaps between neighboring eigenvalues. We recall the convergence result \eqref{eq:factR} from the introduction, which states that the geometric convergence factor is governed by the ratios of the spectral subintervals $[a_1,b_1]$ and $[a_2,b_2]$, $a_1<b_1<0<a_2<b_2$. 
%
%

\medskip

\begin{example}
In Figure~\ref{fig:example1} (top left) we show the relative errors 
$\| F\mathbf{u}_0 - r_n(A)\mathbf{u}_0\|_2/$
$\|F\mathbf{u}_0\|_2$
of the type $(n,n-1)$ rational functions   obtained by RKFIT (dashed red curve) and the two-interval Zolotarev approach (dotted blue) for varying degrees $n=1,2,\ldots,25$. Here the matrix $A$ is defined as $A = L/h^2 - k_\infty^2 I\in\mathbb{R}^{N\times N}$, where $N=150$, $h=1/N$,  $k_\infty=15$, and
\begin{equation}\label{eq:matL}
        L = \small \begin{bmatrix}
         1 & -1 \\
        -1 &  2 & -1 \\
           & \ddots & \ddots & \ddots \\
           &        & -1     & 2      & -1 \\
           &        &        & -1     & 1     
        \end{bmatrix}.
\end{equation}
The matrix $L$ corresponds to a scaled FD discretization of the 1D Laplace operator with homogeneous Neumann boundary conditions. 
The spectral subintervals of $A$ are 
$[a_1,b_1] \approx  [-225, -67.2]$ and 
$[a_2,b_2] \approx  [21.5,   8.98\cdot 10^4]$.
The vector $\mathbf{u}_0\in\mathbb{R}^N$ is chosen at random with normally distributed entries. To compute  the RKFIT approximant~$r_n$ we have used another random training vector $\mathbf{v}$ with normally distributed entries. The corresponding errors $\| F\mathbf{v} - r_n(A)\mathbf{v}\|_2/\|F\mathbf{v}\|_2$ together with the number of required RKFIT iterations are also shown in the plot (solid red curve). For all degrees $n$ at most  $5$ RKFIT  iterations  where required until stagnation occurred. Note that the two RKFIT convergence curves (for the vectors $\mathbf{u}_0$ and $\mathbf{v}$) are very close together, indicating that the random choice for the training vector does not affect much the computed RKFIT approximant. Note further that the RKFIT convergence follows the geometric rate predicted by \eqref{eq:factR} (dotted black curve) very closely initially (up to a degree $n\approx 10$), but then the convergence becomes superlinear. This convergence acceleration is due to the spectral adaptation of the RKFIT approximant.

The spectral adaptation is illustrated in the graph on the top right of Figure~\ref{fig:example1}, which plots the error curve $|f_h(\lambda) - r_{10}(\lambda)|$ of the RKFIT approximant  $r_{10}$ (solid red curve) over the spectral interval of $A$, together with the attained values at the eigenvalues of $A$ (red crosses). In particular, close to $\lambda=0$, there are two eigenvalues at which the error curve attains a relatively small value in comparison to the other eigenvalues farther away (meaning that $r_n$ interpolates $f_h$ nearby). These eigenvalues have started to become ``deflated'' by RKFIT, effectively shrinking the spectral subintervals $[a_1,b_1]$ and $[a_2,b_2]$, and thereby leading to the observed superlinear convergence.

In the bottom of Figure~\ref{fig:example1} we show the poles and residues of the RKFIT approximant $r_{10}$ (left) and the associated continued fraction parameters (right), giving rise to the RKFIT-FD grid. All the involved quantities have been computed using the new \texttt{contfrac} method in the Rational Krylov Toolbox.

\begin{figure}[t]
\hspace*{3mm}\begin{minipage}{16cm}
\hspace*{-3mm}\includegraphics[width=6cm]{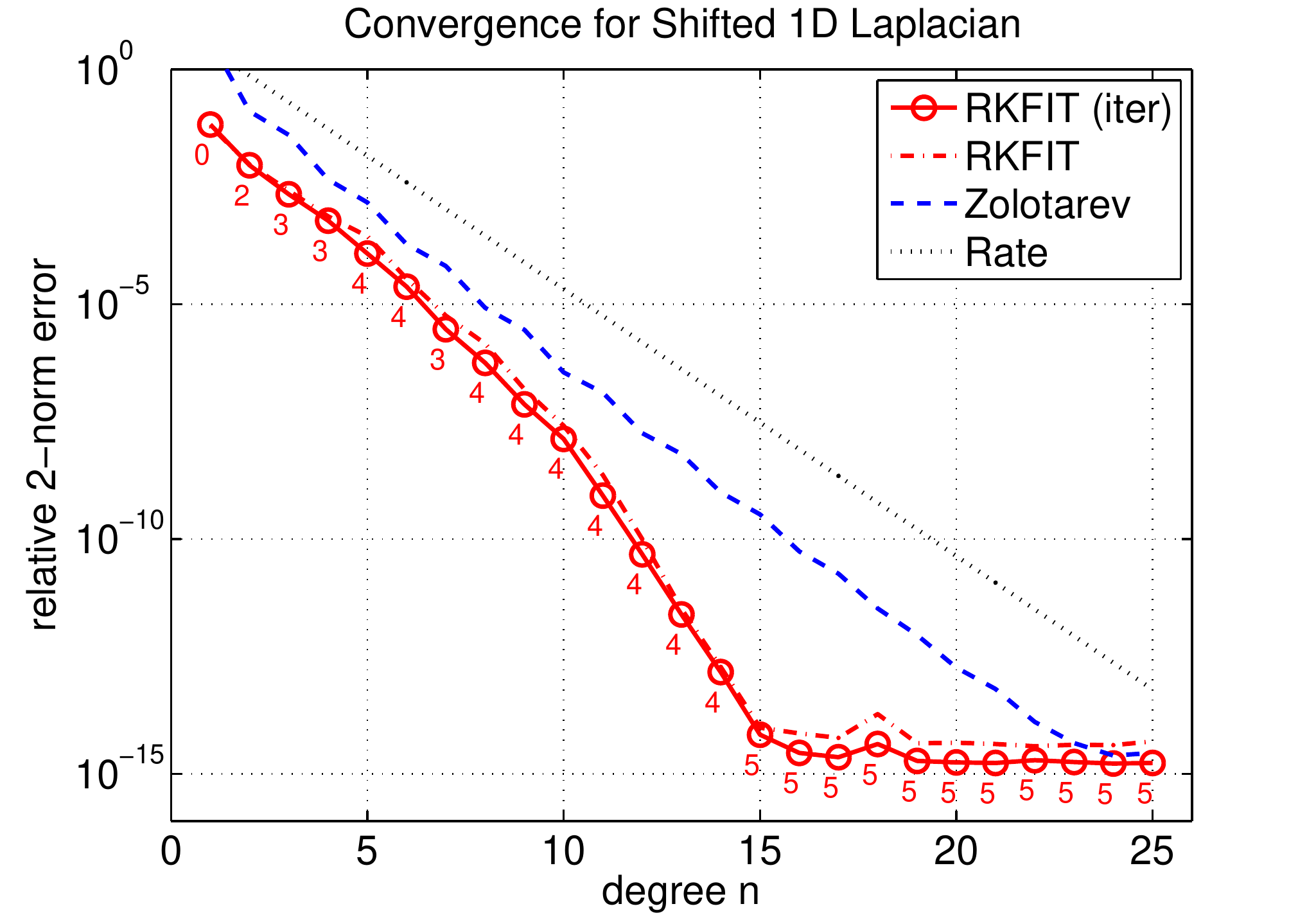}
\hspace*{-4mm}\includegraphics[width=6cm]{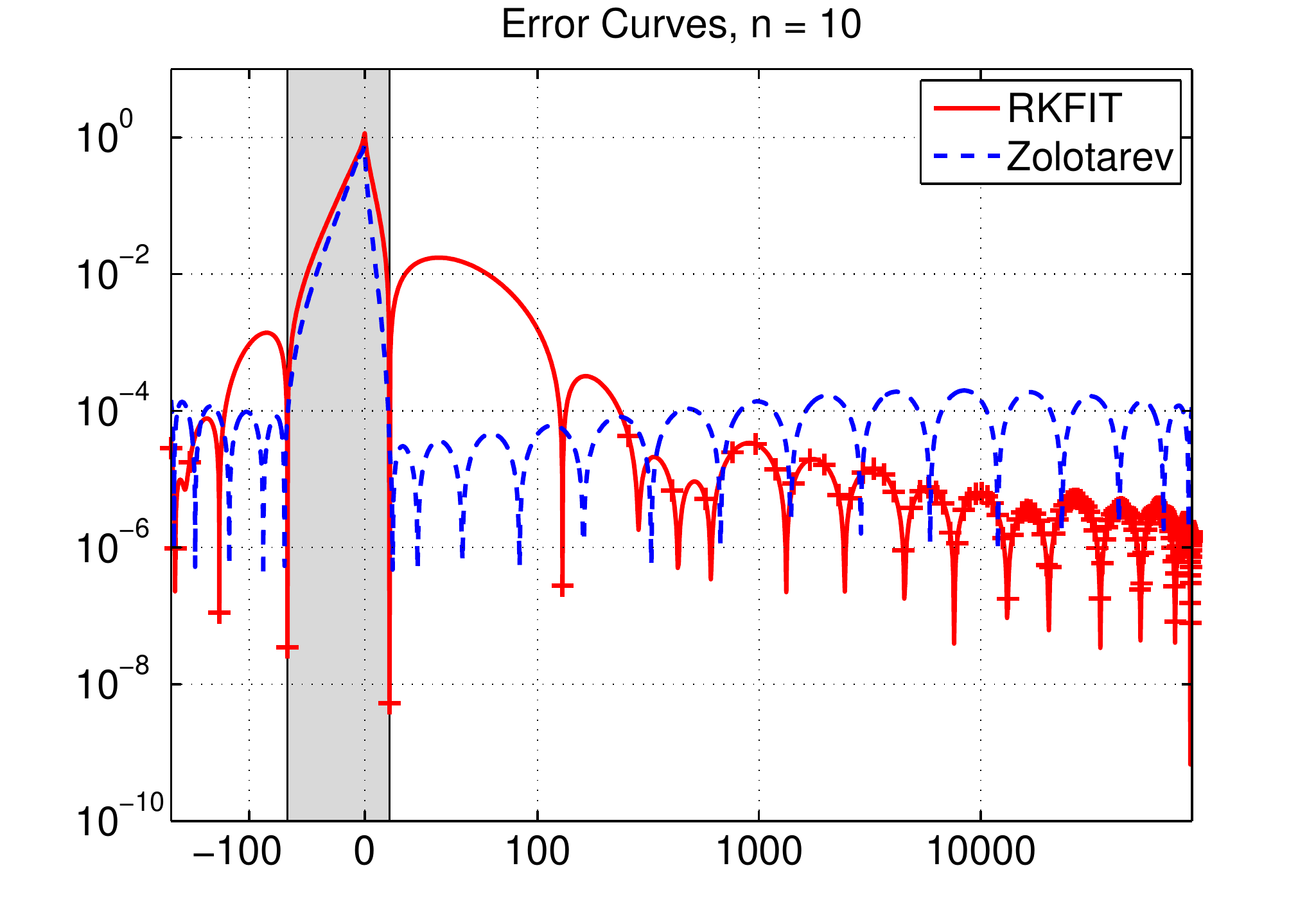}\\[4mm]
\hspace*{-3mm}\includegraphics[width=6cm]{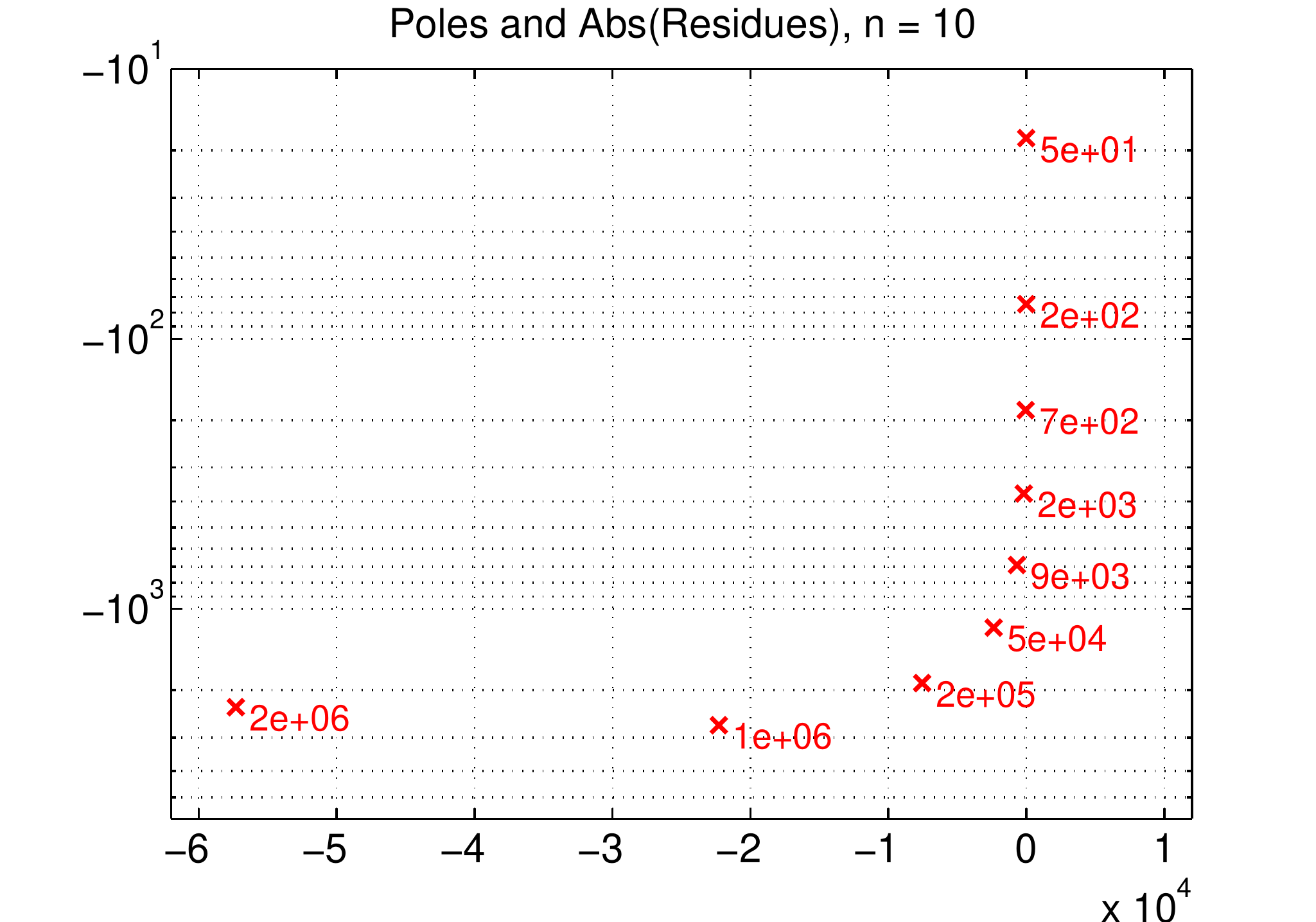}
\hspace*{-4mm}\includegraphics[width=6cm]{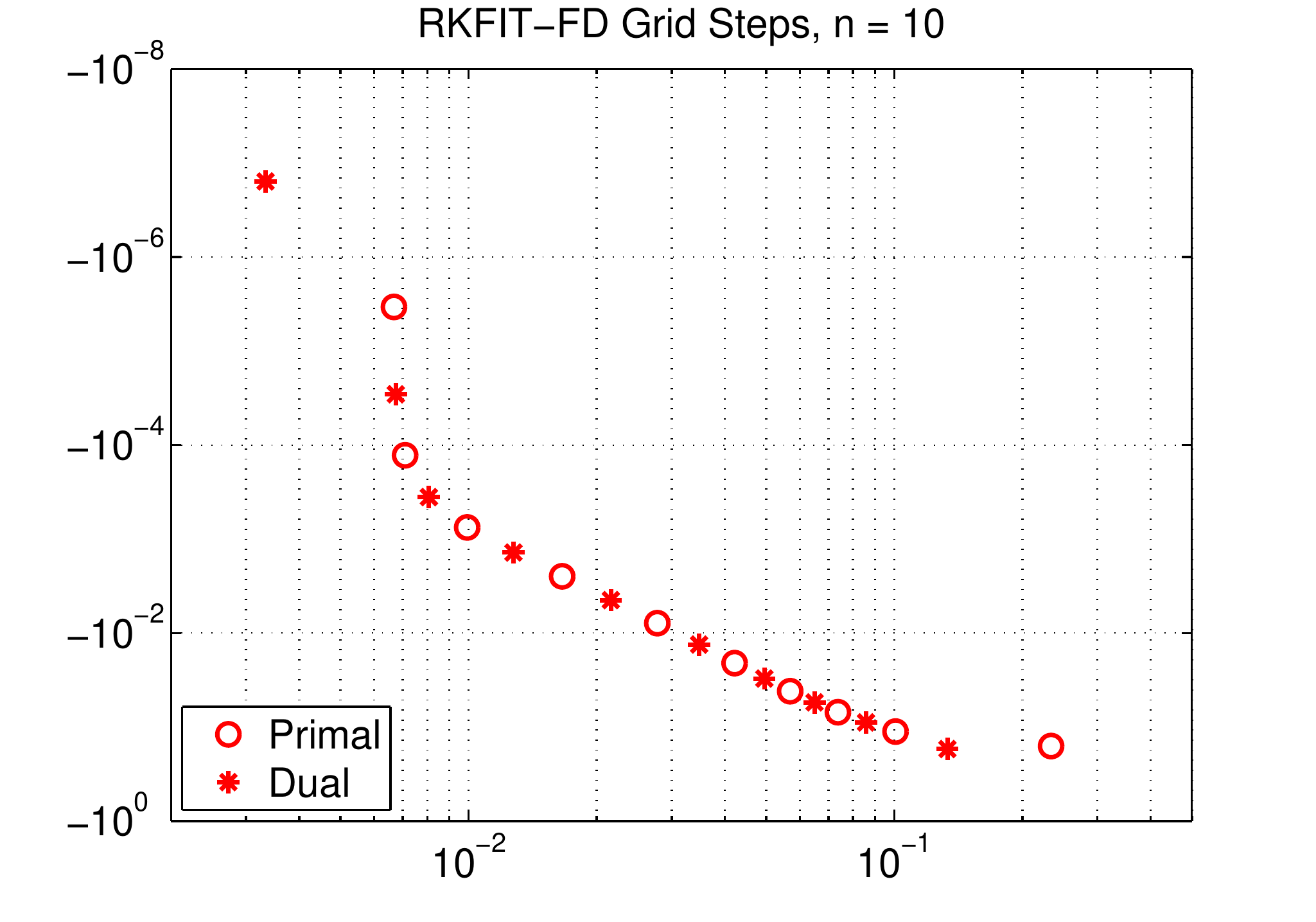} 
\end{minipage}
\begin{minipage}{16cm}
\vspace*{-128.5mm} \hspace*{86mm} {\footnotesize \color{red} $\leftarrow$ spectral adaptation}
\end{minipage}
\caption{Top: Accuracy comparison of RKFIT and Zolotarev approximants for a shifted 1D Laplacian which has a rather coarse spectrum, hence resulting in superlinear RKFIT convergence. The DtN function is $f_h(\lambda)=\sqrt{\lambda + (h\lambda/2)^2}$. The small numbers on the solid red convergence curve on the left indicate the number of required RKFIT iterations. Bottom: The poles and residues of the RKFIT approximant $r_{10}$ (left) and the associated continued fraction parameters (right).}  
\label{fig:example1}
\end{figure}
\end{example}

\medskip

\subsection{Two-interval approximation with dense spectrum}

The superlinear convergence effects observed in the previous example should  disappear when the spectrum of $A$ is dense enough so that, for the order $n$ under consideration, no eigenvalues of $A$ are deflated by interpolation nodes of $r_n$. The next example demonstrates this.

\medskip

\begin{example}\label{ex:example2}
In Figure~\ref{fig:example2} we show the relative errors $\| F\mathbf{u}_0 - r_n(A)\mathbf{u}_0\|_2/\|F\mathbf{u}_0\|_2$ of the type $(n,n-1)$ rational functions obtained by RKFIT and the Zolotarev approach for varying degrees $n=1,2,\ldots,25$. Now  the matrix $A$ corresponds to a shifted 2D Laplacian $A = (L\otimes L)/h^2 - k_\infty^2 I\in\mathbb{R}^{N\times N}$ with $N=150^2$, $h=1/150$,  $k_\infty=15$, and with $L$ defined in \eqref{eq:matL}.
 The special structure of $L$ (and $A$) allows for the use of the 2D discrete cosine transform for computing $F=f_h(A)$. The spectral subintervals of~$A$ are 
$[a_1,b_1] \approx  [-225, -27.7]$ and
$[a_2,b_2] \approx  [21.5,   1.80\cdot 10^5]$.
The vector $\mathbf{u}_0\in\mathbb{R}^{N}$ is chosen at random with normally distributed entries. We also show the relative error of the RKFIT approximant $r_n(A)\mathbf{v}$ with another randomly chosen training vector $\mathbf{v}$, and the number of required RKFIT iterations. As in the previous example there is no big difference in accuracy when evaluating the RKFIT approximant for $\mathbf{u}_0$ or $\mathbf{v}$, however, the number of required RKFIT iterations is slightly higher in this example. 
As the eigenvalues of the matrix $A$ are relatively dense in its spectral interval, we now observe that no spectral adaptation takes place and both the RKFIT and the Zolotarev approximants converge at the rate predicted by \eqref{eq:factR}.

In the bottom of Figure~\ref{fig:example2} we show the grid vectors $\mathbf{u}_j$ satisfying the FD relation~\eqref{eq:Relv} for $n=10$, with the RKFIT-FD grid parameters $h_j$ and $\widehat h_{j-1}$ ($j=0,1,\ldots,10)$ extracted from~$r_{10}$. The entries of $\mathbf{u}_j$ are complex-valued, hence we show the $\log_{10}$ of the amplitude and phase separately. Note how the amplitude decays very quickly as the random signal travels further to the right in the grid, illustrating the good absorption property of this grid.

\begin{figure}[t]
\hspace*{3mm}\begin{minipage}{16cm}
\hspace*{-3mm}\includegraphics[width=6cm]{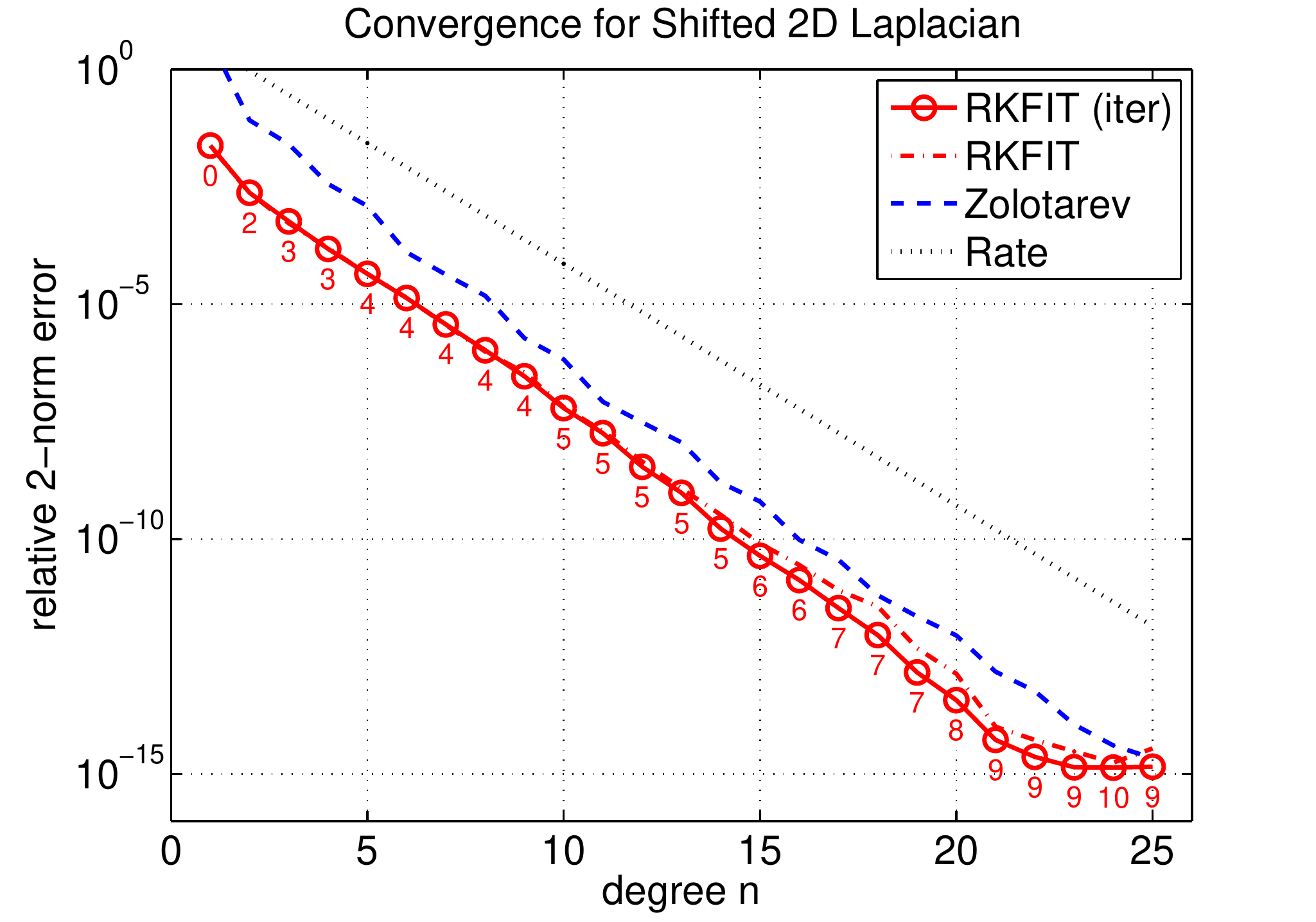}
\hspace*{-4mm}\includegraphics[width=6cm]{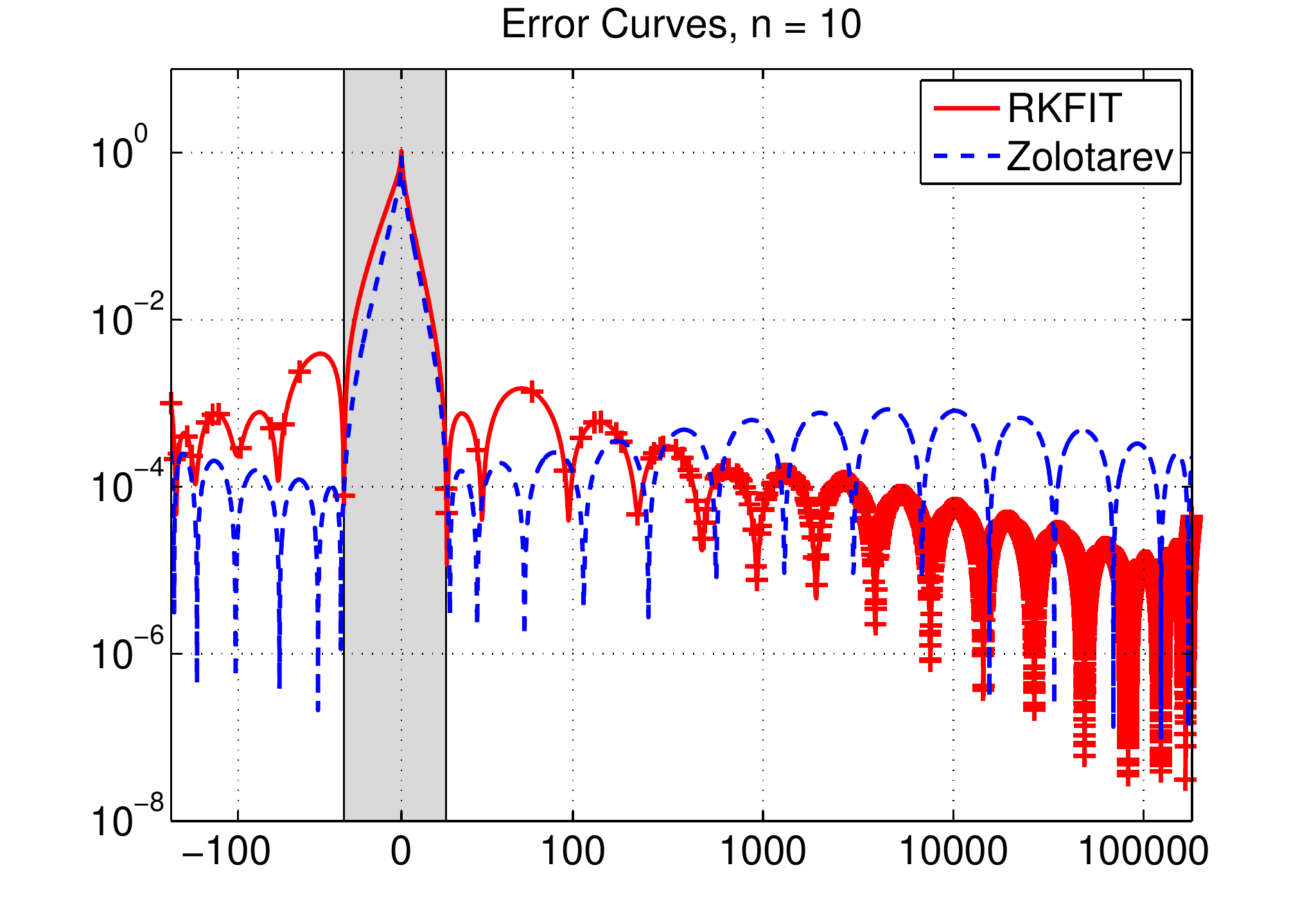}\\[4mm]
\hspace*{-3mm}\includegraphics[width=6.3cm]{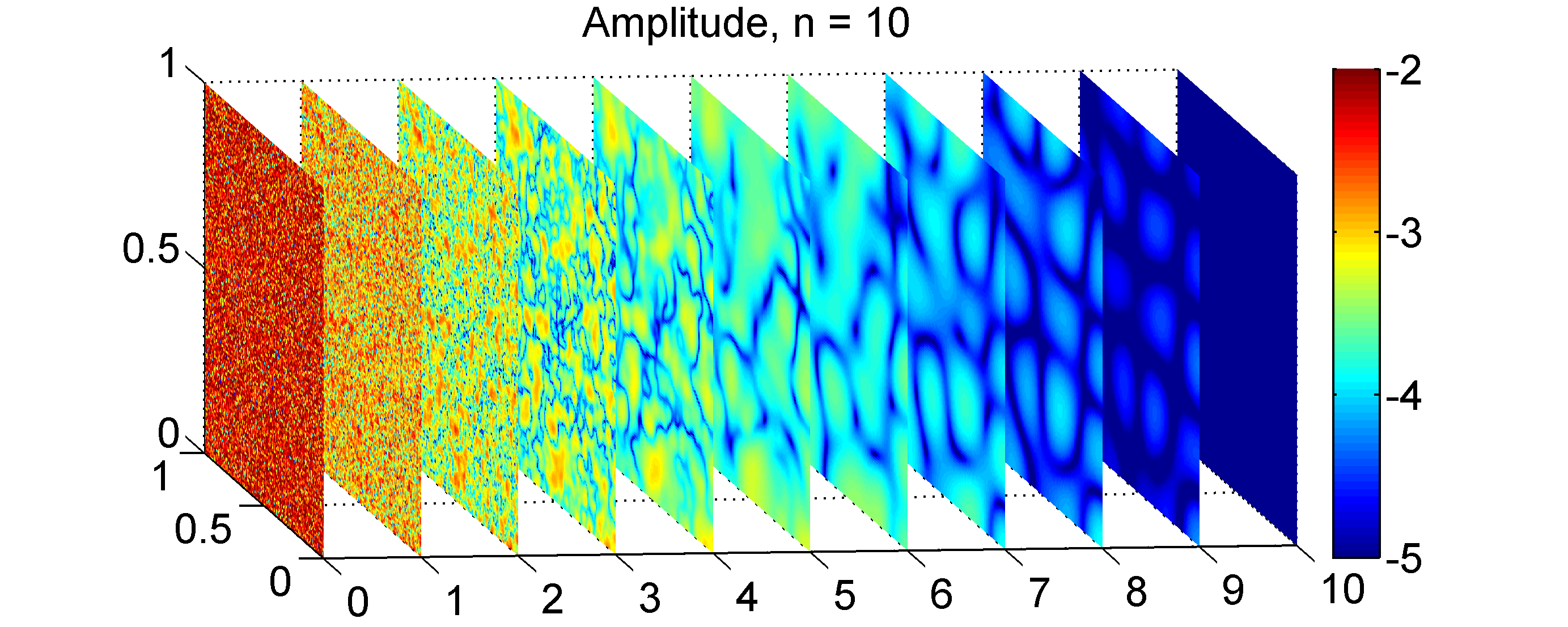}
\hspace*{-7mm}\includegraphics[width=6.3cm]{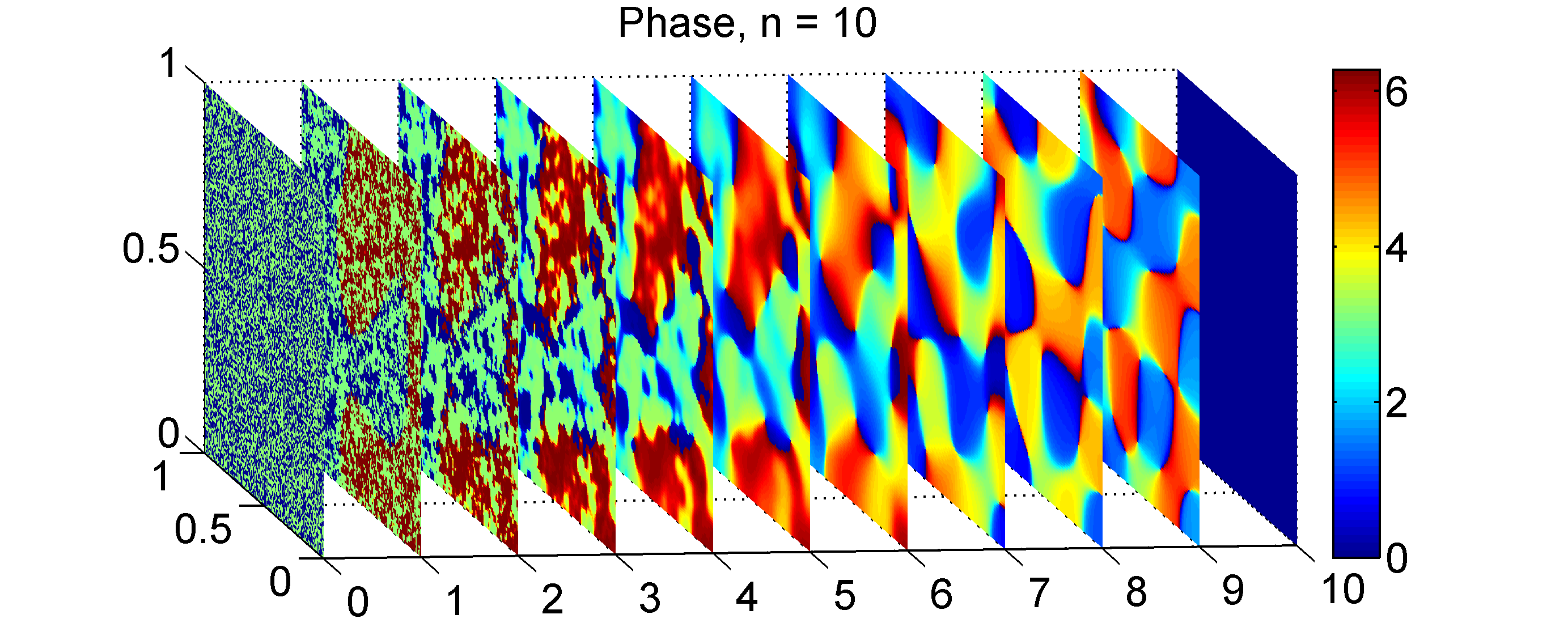}
\end{minipage}
\caption{Top: Comparison of RKFIT and Zolotarev approximants for a shifted 2D Laplacian. Bottom: The $\log_{10}$ of the amplitude and phase of the grid vectors~$\mathbf{u}_j$ $(j=0,1,\ldots,n=10)$. Qualitatively, the poles and residues and the complex grid steps for the associated RKFIT approximant $r_{10}$ look  similar to those in Figure~\ref{fig:example1} and are therefore omitted.}  
\label{fig:example2}
\end{figure}
\end{example}

\subsection{Approximation on an indefinite interval}

In order to \rev{remove the spectral gap $[b_1,a_2]$} from which the previous two examples benefited, we now consider the approximation on an indefinite interval. 

\begin{example}
We approximate $f(\lambda) = \sqrt{\lambda}$ on the indefinite interval 
$[a_1,b_2] = [-225,$ $1.80\cdot 10^5]$.
Note that  $[a_1,b_2]$ is  the same as in the previous Example~\ref{ex:example2}, but without the spectral gap about zero. This problem is of interest as, in the variable-coefficient case, one cannot easily exploit a spectral gap between the eigenvalues of $A$ which are closest to zero. This is because a varying coefficient $c(x)$ can be thought of as a variable shift of the eigenvalues of $A$; hence an eigenvalue-free interval $[b_1,a_2]$ may not always exist. 

To mimic continuous approximation on an interval, we use for $A$ a surrogate diagonal matrix of size $N=200$ having $100$ logspaced eigenvalues in $[a_1,-10^{-16}]$ and $[10^{-16},b_2]$, respectively. The training vector $\mathbf{v}$ is chosen as $[1,1,\ldots,1]^T$. We run RKFIT for degrees $n=1,2,\ldots,25$. 
The relative error of the RKFIT approximants $\| F\mathbf{v} - r_n(A)\mathbf{v}\|_2/\|F\mathbf{v}\|_2$ seems to reduce like $\exp(-\pi\sqrt{n})$; see Figure~\ref{fig:example3} (left). 

We also compare RKFIT to a two-interval Remez-type approximant obtained by using the interpolation nodes of numerically computed best approximants to $\sqrt{\lambda}$ on $[0,1]$, scaling them appropriately, and unifying them for the intervals $[a_1,0]$ and $[0,b_2]$. The number of interpolation nodes on both intervals is balanced so that the resulting error curve is closest to being equioscillatory on the whole of $[a_1,b_2]$. Again the error of the so-obtained Remez-type approximant seems to reduce like $\exp(-\pi\sqrt{n})$.

\begin{remark}
The uniform rational approximation of $\sqrt{\lambda}$ on a semi-definite interval $[0,b_2]$ has been studied by Newman and Vjacheslavov. It is known that the error reduces like $\exp(-\pi\sqrt{2n})$ with the degree $n$; see \cite[Section~4]{PP87}. 
Based on the observations in Figure~\ref{fig:example3} we conjecture that the error of the best rational approximant to $\sqrt{\lambda}$ on an indefinite interval $[a_1,b_2]$ reduces like $\exp(-\pi\sqrt{n})$. 
\end{remark}

\begin{figure}[t]
\hspace*{3mm}\begin{minipage}{16cm}
\hspace*{-3mm}\includegraphics[width=6cm]{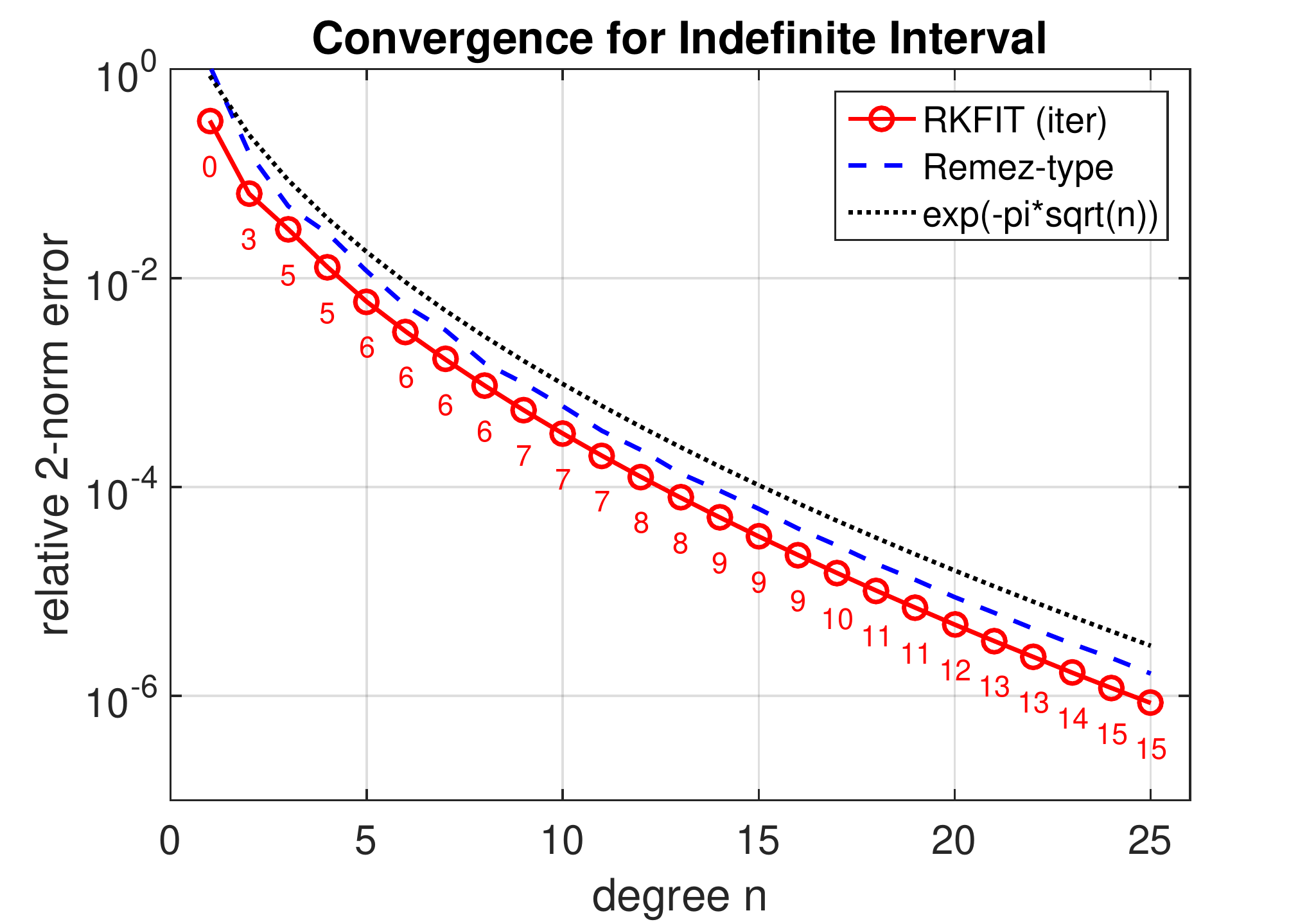}
\hspace*{-4mm}\includegraphics[width=6cm]{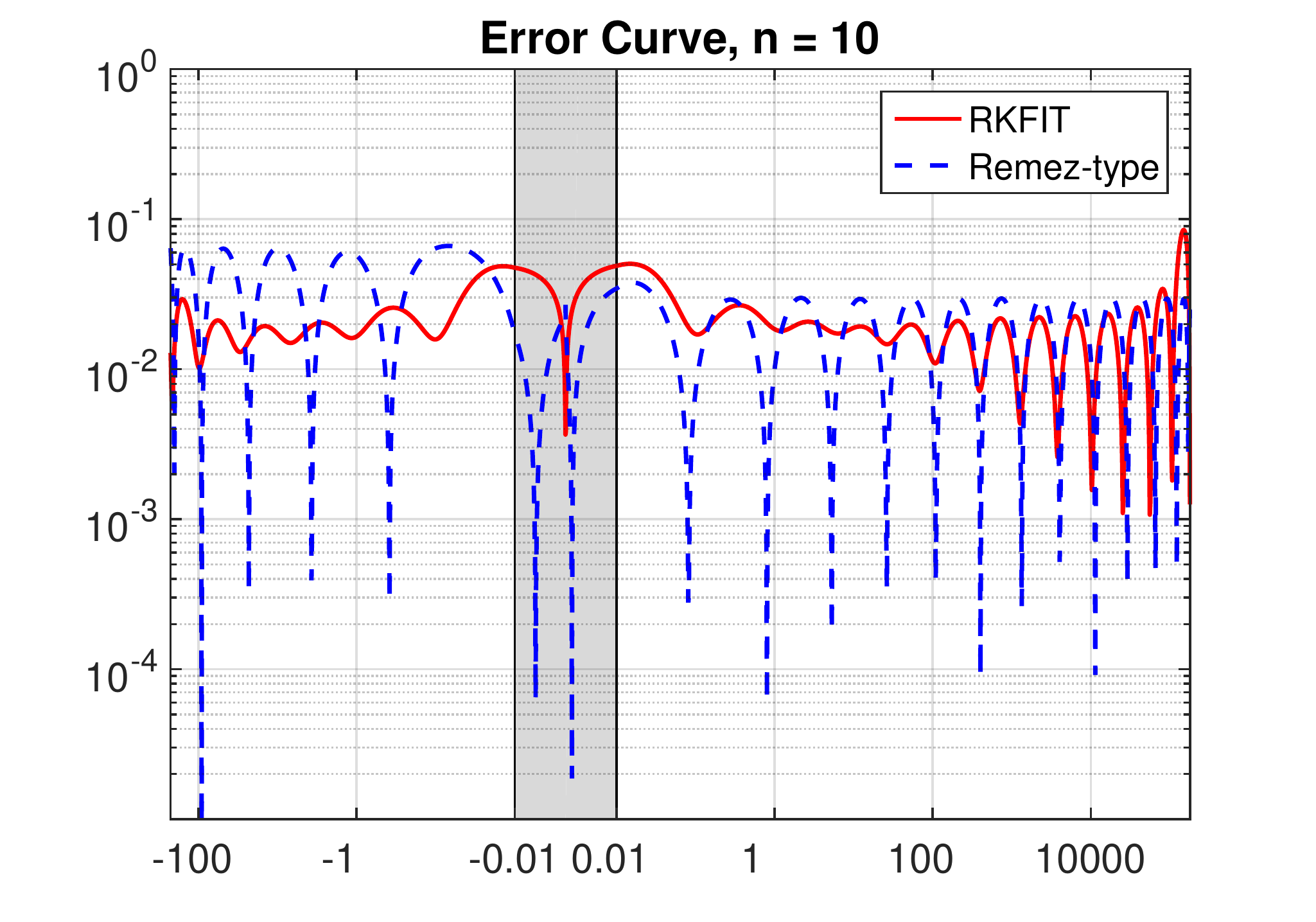}
\end{minipage}
\caption{RKFIT  approximation of $f(\lambda) = \sqrt{\lambda}$  on an indefinite interval $[a_1,b_2]$, $a_1<0<b_2$, compared to a two-interval Remez-type approximant.  Qualitatively, the poles/residues and the complex grid steps associated with $r_{10}$ look similar to those in Figure~\ref{fig:example1} and are therefore omitted.}  
\label{fig:example3}
\end{figure}
\end{example}

\section{Numerical tests: variable-coefficient case}\label{sec:conv2}

We now consider  a variable-coefficient function~$c$ motivated by a  geophysical seismic exploration setup as shown in Figure~\ref{fig:drawing}. Here a pressure wave signal of a single frequency is emitted by an acoustic transmitter in the Earth's subsurface, travels through the underground, and is then logged by receivers on the surface. From these measurements geophysicists try to infer variations in the wave speed   to draw conclusions about the subsurface composition. The computational domain of interest is a three-dimensional portion of the Earth and  we might have knowledge about the sediment layers below this domain, i.e., for $x\geq 0$ in Figure~\ref{fig:drawing}. While the acoustic waves in $x\geq 0$  may not be of interest on their own, the layers might cause wave reflections back into the computational domain and hence need be part of the model. 


\smallskip

\begin{figure}
\begin{minipage}{16cm}
\hspace*{30mm}\includegraphics[width=6cm]{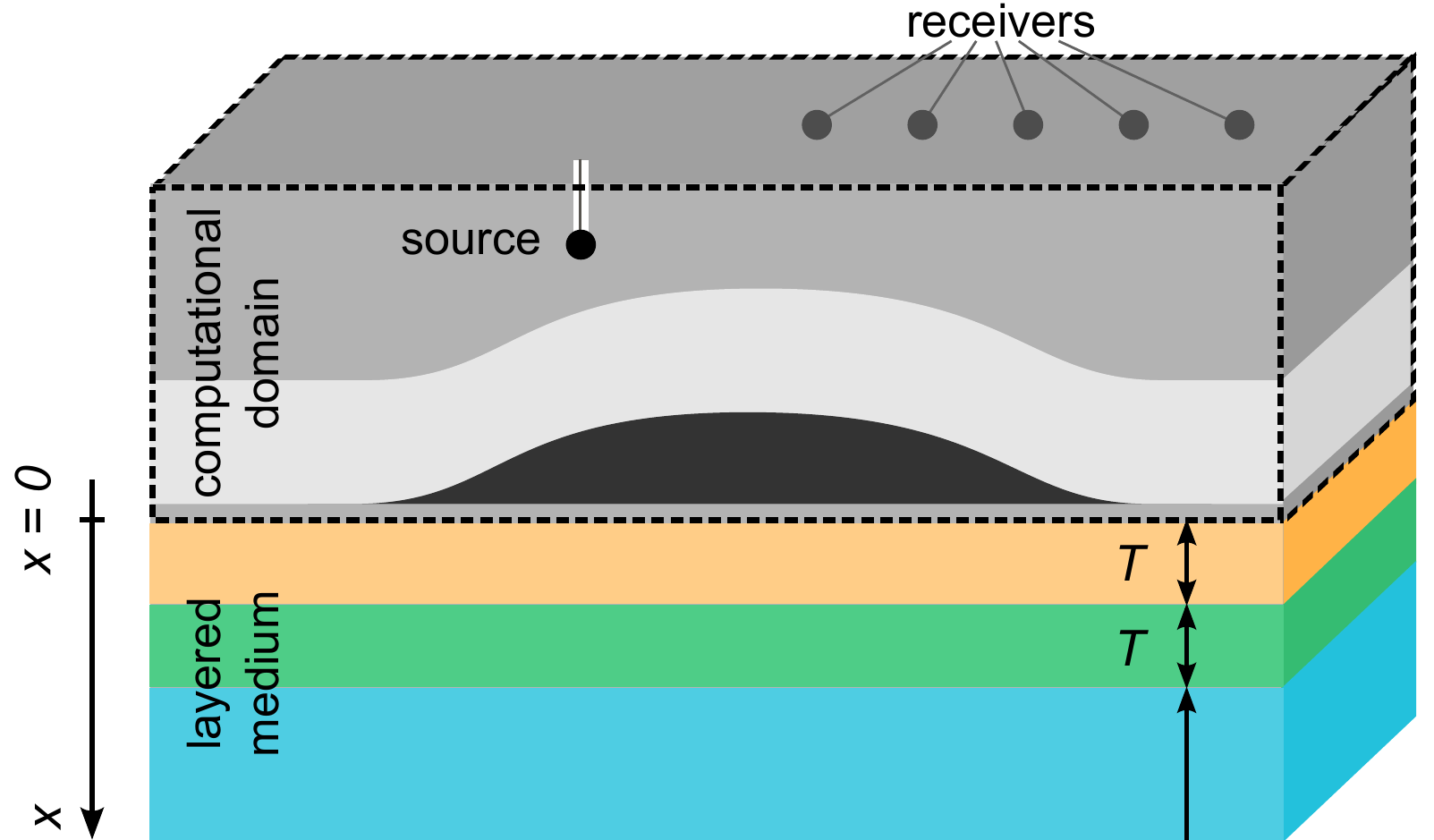}
\vspace*{-0mm}
\end{minipage}
\caption{Typical setup of a seismic exploration of the Earth's subsurface. It is of practical interest to compress the layered medium in $x\geq 0$ into a single PML with a small number of grid points.}  
\label{fig:drawing}
\end{figure}

\begin{figure}[t]
\hspace*{3mm}\begin{minipage}{16cm}
\hspace*{-13mm}\includegraphics[width=13.3cm]{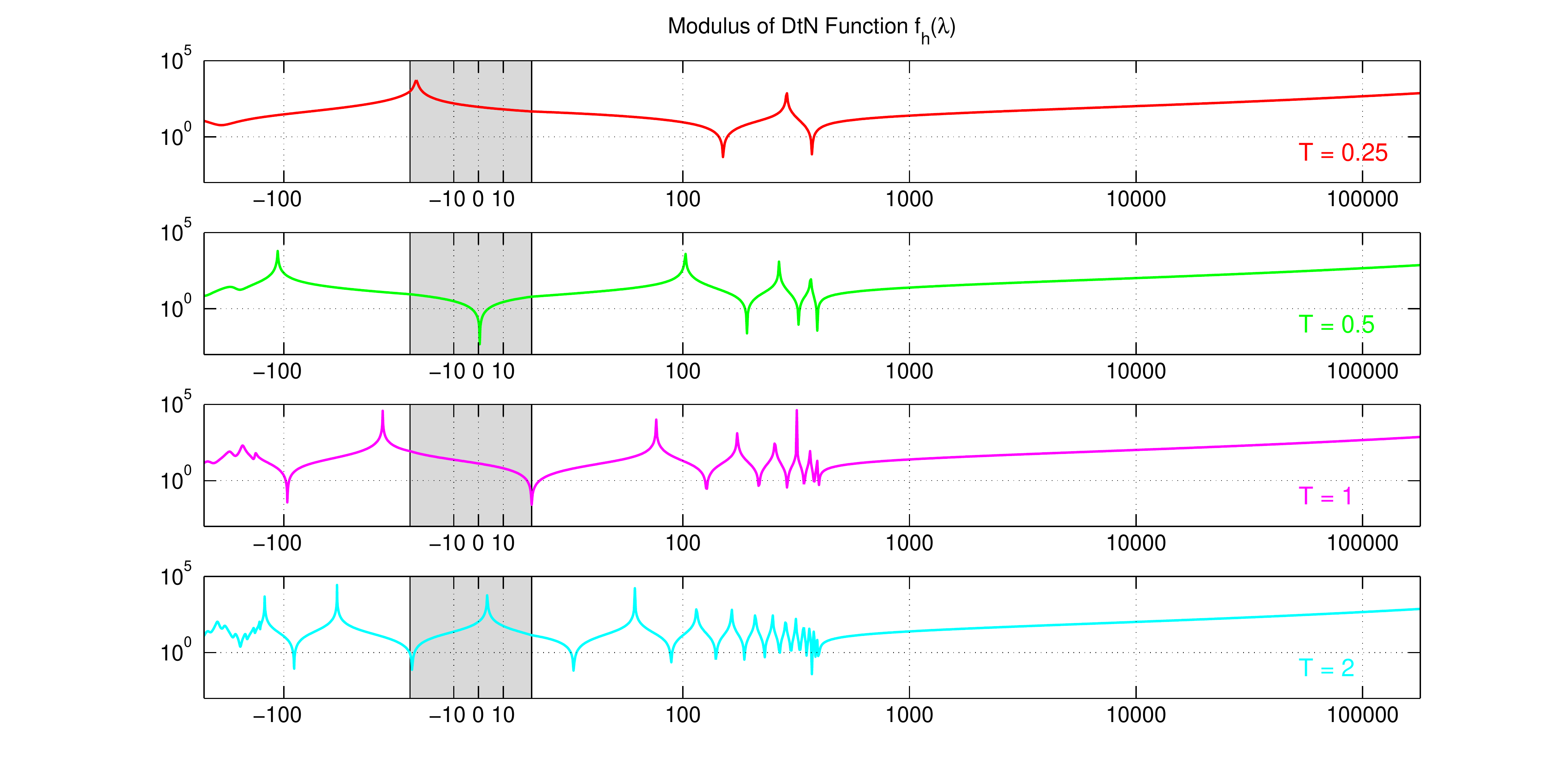}\\
\hspace*{-3mm}\includegraphics[width=6cm]{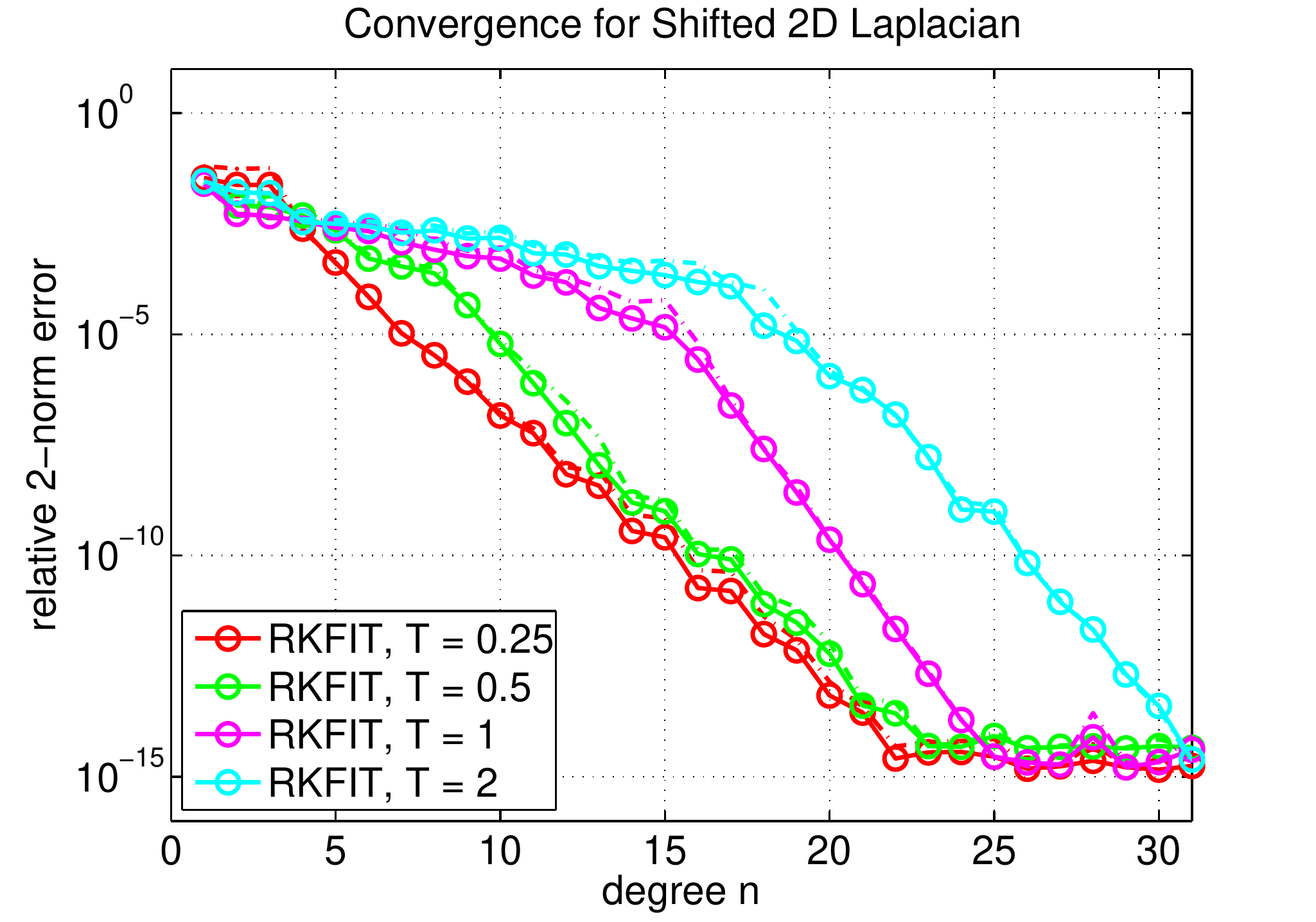}
\hspace*{-3mm}\includegraphics[width=6cm]{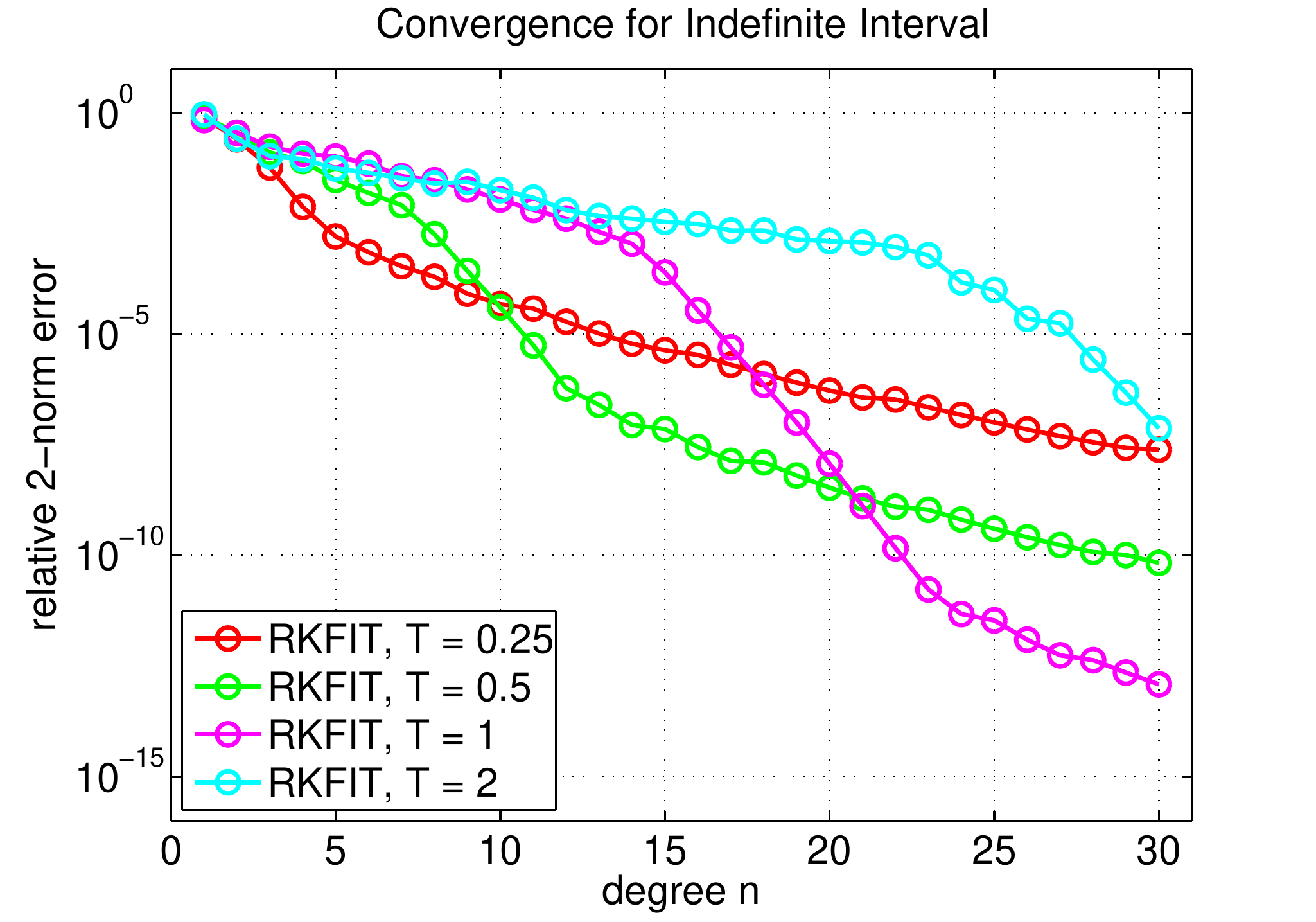}
\end{minipage}
\caption{Top: The four panels show the modulus of the discrete variable-coefficient DtN function $f_h$ for varying thickness $T$ of the two finite layers. 
Bottom: The two plots show the RKFIT convergence for approximating $f_h(A)\mathbf v$ when $A$ is a shifted 2D Laplacian (left) and a diagonal matrix with dense eigenvalues in the same spectral interval (right), respectively.}  
\label{fig:example45}
\end{figure}

\begin{example}\label{ex:vc1}
At the $x=0$ interface of the computational domain, shown in Figure~\ref{fig:drawing}, we assume to have a 2D Laplacian $A = (L\otimes L)/h^2 - k_\infty^2 I$ with $L$ defined in \eqref{eq:matL}, and $N=150^2$, $h=150$, and $k_\infty = 15$. Now the function $f_h$ of interest is \eqref{eq:vcdtn}, with the coefficients $c_j$ obtained by discretizing the piecewise-constant coefficient function $c$ 
which equals $-400$ on $[0,T)$, $+125$ on $[T,2T)$ and $0$ on $[2T,\infty)$. 
\rev{The} thickness of the two finite layers $T$  is varied in $\{0.25,0.5,1,2\}$. For each thickness $T$, the four panels in the top of Figure~\ref{fig:example45} show the modulus of $f_h$ over the spectral subintervals $[a_1,b_1]$ and $[a_2,b_2]$ of $A$, glued together with the gray linear region $[b_1,a_2]$. It becomes apparent that with increasing $T$ the function $f_h$ exhibits more poles on or nearby the spectral interval of $A$, indicated by the upward spikes. 

The convergence of the RKFIT approximants for increasing degree~$n$ is shown in Figure~\ref{fig:example45} on the bottom left. For each thickness $T$ there are two curves very nearby: a solid curve showing the relative 2-norm approximation error for $F\mathbf{v}$ (where $\mathbf{v}$ is a random training vector) and a dashed curve for $F\mathbf{u}_0$ (where $\mathbf{u}_0$ is another random testing vector).  We observe that RKFIT converges very robustly for this piecewise constant-coefficient problem. Similar behavior has been observed in many numerical tests with other offset functions $c$. We refer to the example collection of the Rational Krylov Toolbox which contains further examples. The codes for producing our examples are available online and can easily  be modified  to other coefficient functions.
\end{example}

\begin{example}\label{ex:vc2}
Here we consider a diagonal matrix $A$ with the same indefinite spectral interval as the matrix in the previous example but with dense eigenvalues, namely  $100$ logspaced eigenvalues in $[a_1,-10^{-16}]$ and $[10^{-16},b_2]$, respectively. The convergence is shown on the bottom right of 
Figure~\ref{fig:example45}. Again the RKFIT behavior is very robust even for high approximation degrees $n$, but compared to the above Example~\ref{ex:vc1} the convergence is  delayed, indicating that spectral adaptation has been prevented here. 
\end{example}


\section{Discussion and conclusions}\label{sec:disc}
An obvious alternative to our grid compression approach in the two examples of section~\ref{sec:conv2} would be to use an efficient discretization method on $c$'s support, and then to append it to the constant-coefficient PML of \cite{DGK16}. In principle such an approach requires at least the integer part of 
$
 \text{N}={\pi}^{-1} \int_{0}^H {\sqrt{k_\infty^2-c(x)}} \d x
$ 
discretization points according to the Nyquist sampling rate, where $H$ is the total thickness of $c$'s support.
 In fact, the classical spectral element method  (SEM) with polynomial local basis requires at least $\frac{\pi}{2}\text{N}$ grid points \cite{ainsworth2009dispersive}. (The downside of SEM compared to our FD approach is its high linear solver cost per unknown caused by the dense structure of the resulting linear systems.) The following table shows the minimal number of grid points required for discretizing the two finite layers in the examples of section~\ref{sec:conv2}, depending on the layer thickness $T$, as well as the number of RKFIT-FD grid points to achieve a relative accuracy of $10^{-5}$ for the same problem:
 
 \medskip

\begin{small}
 \hspace*{15mm}\begin{tabular}{|l|c|c|c|c|}
 \hline 
    & $T=0.25$ & $T=0.5$ & $T=1$ & $T=2$ \\ 
 \hline 
Nyquist minimum $\text{N}$ & 8.75 & 17.5 & 35 & 70 \\ 
 \hline 
SEM minium $\frac{\pi}{2}\text{N}$ & 13.7 & 27.5 & 55.0 & 110.0 \\ 
 \hline 
RKFIT-FD (Example~\ref{ex:vc1}) & 8 & 10 & 16 & 19 \\ 
 \hline 
RKFIT-FD (Example~\ref{ex:vc2}) & 14 & 11 & 17 & 28 \\ 
 \hline 
 \end{tabular}  
\end{small}

\bigskip

\noindent Although we also observe with RKFIT-FD a tendency that the DtN functions become more difficult to approximate when the layer thickness increases (an increase of the coefficient jumps between the layers will have a similar effect), the  number of required grid points can be significantly smaller than the Nyquist limit $\text{N}$. A possible explanation for this phenomenon is RKFIT's ability to adapt to the spectrum of~$A$, not being slowed down in convergence by singularities of the DtN function well separated from the eigenvalues of $A$. In the appendix we analyze this phenomenon.


\medskip

\noindent \textbf{Acknowledgements.}  \rev{We thank Ralf Hiptmair and the anonymous referees for constructive comments that have significantly improved the presentation.}
 Druskin was partially supported by an  
Air Force Office of Scientific Research (AFOSR) grant FA~955020-1-0079 and a National Science Foundation (NSF) grant DMS-2110773. 
 G\"{u}ttel was partially supported by The Alan Turing Institute, Engineering and Physical Sciences Research Council (EPSRC) grant EP/W001381/1. 
\rev{Knizhnerman was supported by the Moscow Center of Fundamental and Applied Mathematics (Agreement 075-15-2019-1624 with the Ministry of Education and Science of the Russian Federation).}


\vspace*{-4mm}

\appendix
\section{Nyquist limit-type criterion for rational approximation}\label{sec:nyquist}
The top-four panels in Figure~\ref{fig:example45} suggest that the DtN function $f_h$, specified in \eqref{eq:vcdtn}, develops more and more poles on the real axis as the thickness of the finite layers increases. {These poles are also known as scattering resonances. 
In order to analyze this behavior, we consider a two-layer waveguide problem with piecewise constant wave numbers similar to the one in Figure~\ref{fig:waveguide}, but now in the continuous setting without any FD approximation. 
This problem is governed by the equations
\begin{eqnarray*}
u''(x) = (\lambda+c) u(x)\ \ \text{for}\  \ x\in [0,T), \qquad
u''(x) = \lambda u(x)\ \ \text{for}  \ \ x\in [T,\infty),
\end{eqnarray*}
with given $u(0)=u_0$ and the decay condition $u(x)\to 0$ as $x\to \infty$. Here, $T$ is the thickness of the first layer with an offset coefficient $c$. In terms of the Helmholtz equation, a value $c=-k_0^2<0$ means that the wave number on the first layer is larger than on the second, whereas $c> 0$ means that the wave number on the first layer is smaller than on the second. If $c=0$ we have a homogeneous infinite waveguide. 

Our aim is to solve for $u$ explicitly and to determine a formula for the DtN function $f$ satisfying $f(\lambda)u_0 = -u'(0)$. For $x\in [0,T]$ we have 
\begin{eqnarray*} 
u(x) = \alpha e^{x \sqrt{\lambda+c}} + (u_0-\alpha) e^{-x \sqrt{\lambda+c}}
2\alpha \sinh\big( x \sqrt{\lambda+c}\big) + e^{-x \sqrt{\lambda+c}} u_0,
\end{eqnarray*}
where the square roots are understood as the analytical continuation through the upper half plane from the axis $\lambda > -c$. 
For $x\in [T,\infty)$ we require a decaying solution, hence 
$u(x) = \beta e^{-x \sqrt{\lambda}}$ there.   
By continuity of $u(x)$ at $x=T$  we have 
\[
        \beta  =   \big( 2\alpha  \sinh\big(T\sqrt{\lambda+c}\big) +  e^{-T\sqrt{\lambda+c}}u_0\big)\cdot e^{T\sqrt{\lambda}}.
\]
By continuity of $u'(x)$ at $x=T$ we further require 
\[
        \sqrt{\lambda+c}\cdot\big( 2\alpha \cosh(T\sqrt{\lambda+c})  -  e^{-T\sqrt{\lambda+c}} u_0\big)
        =  - \beta \sqrt{\lambda}\cdot e^{-T\sqrt{\lambda}},
\]
hence 
\[
        \sqrt{\lambda+c}\cdot\big( 2\alpha \cosh(T\sqrt{\lambda+c})  -  e^{-T\sqrt{\lambda+c}} u_0\big)
        =  -  \big( 2\alpha  \sinh\big(T\sqrt{\lambda+c}\big) +  e^{-T\sqrt{\lambda+c}}u_0\big)\cdot \sqrt{\lambda},
\]
from which $\alpha$ can be determined as
\[
         \alpha
        = \frac{u_0}{2} \cdot \frac{\big( \sqrt{\lambda+c} - \sqrt{\lambda}  \big)e^{-T\sqrt{\lambda+c}}}{\sqrt{\lambda+c} \cosh(T\sqrt{\lambda+c})  + \sqrt{\lambda} \sinh\big(T\sqrt{\lambda+c}\big)}. 
\]
Note that $\alpha = \alpha_\lambda$ is a function of $\lambda$. Using the fact that $u'(0) = (2\alpha_\lambda - u_0) \sqrt{\lambda + c}$, the DtN function $f$ satisfying $f(\lambda) u_0 = -u'(0)$ is given as
\begin{equation}\label{eq:flam}
        f(\lambda) 
        =  \frac{\sqrt{\lambda+c}\cdot \sinh(T\sqrt{\lambda+c})  + \sqrt{\lambda}\cdot \cosh\big(T\sqrt{\lambda+c}\big)}{\sqrt{\lambda+c}\cdot \cosh(T\sqrt{\lambda+c})  + \sqrt{\lambda}\cdot \sinh\big(T\sqrt{\lambda+c}\big)} \cdot \sqrt{\lambda+c}.
\end{equation}
A plot of this function  for two different parameter choices $T=5$ and $c = \pm  9$ is shown in Figure~\ref{fig:cntdtn}. We observe that this function is smooth over the whole real axis when $c\geq 0$, while it develops singularities when  $c<0$. The following lemma shows that the number of real poles is proportional to $c$ and $T$.

\begin{figure}[t]
\hspace*{3mm}\hspace*{-4mm}\includegraphics[width=6cm]{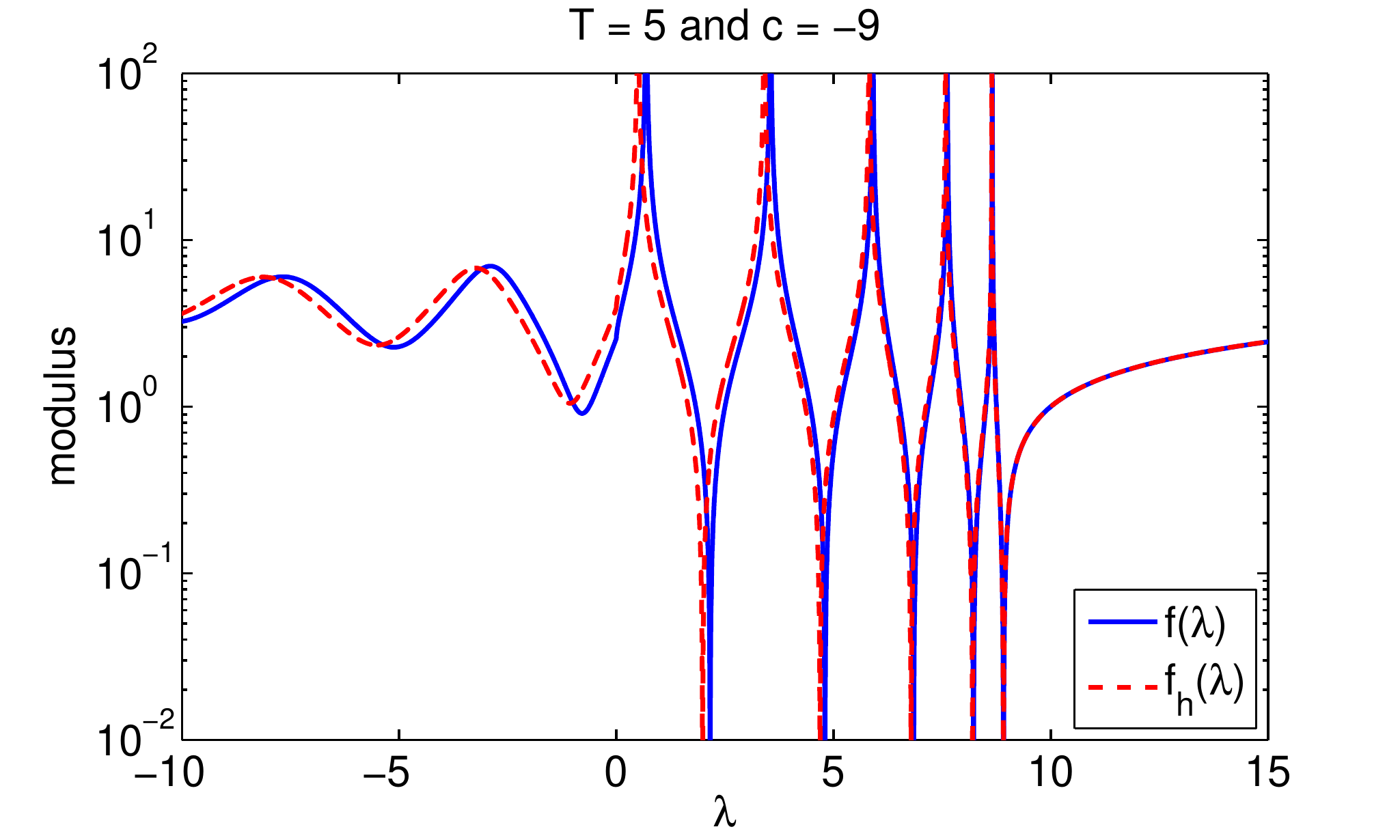}\includegraphics[width=6cm]{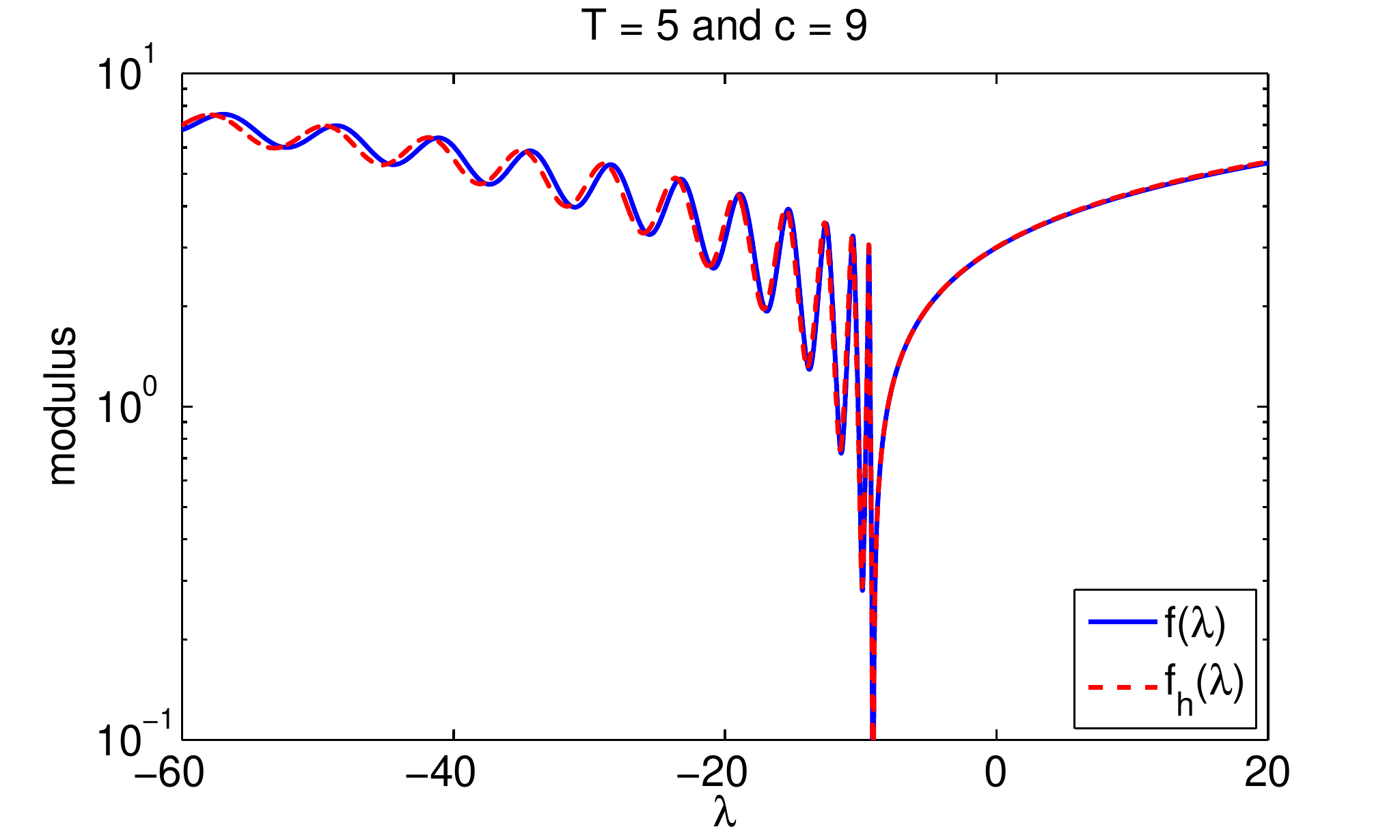}\vspace*{-5mm}
\hspace*{3mm}\begin{minipage}{16cm}
\vspace*{-55.5mm} \hspace*{5mm} {\footnotesize \color{blue} $5$ real poles $\rightarrow$\hspace*{67mm} no real poles}
\end{minipage}
\vspace*{-3.5mm}
\caption{The DtN function $f$ defined in \eqref{eq:flam}, as well as its  discrete counterpart \eqref{eq:vcdtn}, for two different  choices of the parameters $(T,c)$.}  
\vspace*{-2mm}
\label{fig:cntdtn}
\end{figure}

\begin{lemma}\label{lem:nyquist}
The function $f$ defined in~\eqref{eq:flam} can be analytically continued from $\lambda > \max \{ 0,-c\}$ through the upper half plane to the whole real axis except for two ramification points $\lambda = 0$ and $\lambda = -c$ and possibly a finite number of poles. 
For $c > 0$, the function $f$ has no poles on the real axis. 
For $c<0$, the function $f$ has $\left\lfloor \frac{T \sqrt{-c}}{\pi} \right\rfloor + q$ real poles, where $q\in \{0,1\}$, all located in the interval $(0,-c)$. 
\end{lemma}

\begin{proof}
We investigate the roots of the denominator  
$g(\lambda)=\sqrt{\lambda+c}\cdot \cosh(T\sqrt{\lambda+c})  + \sqrt{\lambda}\cdot \sinh(T\sqrt{\lambda+c})$.
We first consider the case $c<0$ and argue that there are no real roots of $g$ outside $[0,-c]$.  For $\lambda < 0$, the factors $\sqrt{\lambda + c}$ and $\sqrt{\lambda}$ are purely imaginary and nonzero, while $\cosh(T\sqrt{\lambda+c}) = \cos(Tz)$ is purely real and $\sinh(T\sqrt{\lambda+c}) = i\sin(Tz)$ purely imaginary (here and throughout the proof $z=\mathrm{imag}(\sqrt{\lambda+c})$). Hence, $\lambda$ can only be a root of $g$ if $\cos(Tz) = \sin(Tz) = 0$,  but this cannot happen as $\cos(\cdot)$ and $\sin(\cdot)$ do not have any roots in common. A similar argument shows that there are no roots of $g$ for $\lambda>-c$. 

For $\lambda\in (0,-c)$, $z = \mathrm{imag}(\sqrt{\lambda+c})$ varies in $(0,\sqrt{-c})$ and we want to count the number of roots of the purely imaginary function 
$h(z) = g(\lambda) = iz\cos(Tz) + \sqrt{z^2+c}\cdot\sin(Tz)$ 
on that interval. Consider the subintervals $I_k=((k-1)\pi/T,k\pi/T]$ for $k=1,2,\ldots,K = \lfloor T\sqrt{-c}/\pi \rfloor$. Then on the first half of each $I_k$ the function $\mathrm{imag}(h)$ is strictly positive (or negative), while on the second half it is strictly monotonically decreasing (increasing) with a sign change. Therefore each $I_k$ contributes exactly one root of $h$.  The final interval $(K\pi/T, \sqrt{-c})$ may or may not contain a further root of $h$. By the same argument one shows that the roots of the numerator of $f$ are located on the first half's of $I_k$, and hence the roots of the denominator do not cancel out.

For $c\geq 0$ one argues similarly to the first part of the proof that the denominator function $g$ has no roots for all real values of $\lambda$. 
 \hfill $\square$
\end{proof}

To interpret this result in terms of the indefinite Helmholtz equation $(\partial_{yy}   + \partial_{zz}) u + (k_\infty^2 - c(x))u = 0$ for $c<0$, first note that the DtN function \eqref{eq:flam} does not depend on $k_\infty$, but merely on the  offset $c$. We may therefore  set $k_\infty = 0$, in which case the wave number on the first finite layer is simply $k = \sqrt{-c}$. Furthermore, $\ell = 2\pi/k = 2\pi/\sqrt{-c}$ is the corresponding wavelength. Using this notation, Lemma~\ref{lem:nyquist} states that $f$ has $\approx 2T/\ell$ poles on the real axis, that is, \emph{two real poles per wavelength!}

Although Lemma~\ref{lem:nyquist} is stated for the continuous waveguide problem, discrete DtN functions $f_h$ seem to have poles very close to those of their continuous counterparts $f$. An example is shown in Figure~\ref{fig:cntdtn} (dashed red curve), which corresponds to \eqref{eq:vcdtn} with ``piecewise'' constant coefficients $c_j$ and $h=0.05$.

Returning to the RKFIT convergence, we observed in the experiments in section~\ref{sec:conv2} that the minimal number $n$ of RKFIT-FD grid points required to achieve convergence does not seem to be directly linked to the Nyquist criterion. Although $f_h$ may have a large number $\text{N}$ of singularities on the spectral interval of $A$, RKFIT's spectral adaptation capabilities mean that $r_n$ does not need to resolve them all, and therefore the degree $n$ can be significantly smaller than $\text{N}$. Although Lemma~\ref{lem:nyquist} effectively states a Nyquist-type criterion for the layered waveguide, from a rational approximation point of view RKFIT-FD grids can outperform it in case of a favourable spectral distribution of the matrix $A$.

\bibliography{paper}

\end{document}